\documentclass[a4paper,12pt]{article}

\usepackage{natbib}  
\usepackage{color} 
\usepackage{hyperref} 
\usepackage{pstricks}
\usepackage{amssymb}
\usepackage{amsfonts,graphicx,amsmath,amsthm,color,nicefrac,enumerate, layout, bbm, listings}

\definecolor{lightblue}{rgb}{0.20,0.40,0.70}

\hypersetup{
breaklinks,
colorlinks=true,
linkcolor=lightblue,
urlcolor=lightblue,
citecolor=lightblue}

\renewcommand{\cite}{\citet}

\theoremstyle{plain}
\newtheorem{theorem}{Theorem}[section]                                          
\newtheorem{proposition}[theorem]{Proposition}                          
\newtheorem{lemma}[theorem]{Lemma}

\theoremstyle{definition}

\newtheorem{sett}[theorem]{Setting}
\theoremstyle{remark}
\newtheorem{remark}[theorem]{Remark}

\makeatletter \@addtoreset{equation}{section} \makeatother

\newcommand{\Z}{\mathbb{Z}}

\newcommand{\R}{\mathbb{R}}

\newcommand{\W}{{W}}
\newcommand{\N}{\mathbb{N}}
\newcommand{\E}{\mathbb{E}}
\newcommand{\und}{\underline}
\newcommand{\Id}{\mathrm{Id}}
\newcommand{\F}[0]{\ensuremath{\mathcal{F}}}
\newcommand{\one}[0]{\ensuremath{\mathbbm{1}}}
\newcommand{\smallsum}{\textstyle\sum}
\renewcommand{\P}{\mathbb{P}}

\newcommand{\lf}{\lfloor}
\newcommand{\rf}{\rfloor}

\newcommand{\Y}{\mathcal{X}}
\newcommand{\vt}{c_2}
\newcommand{\kt}{\epsilon}

\title
{Strong convergence for explicit 
space-time discrete 
numerical approximation for 
2D stochastic Navier-Stokes equations}

\author{Sara Mazzonetto 
\smallskip
\\
\small{Universit{\"a}t Duisburg-Essen, Fakult{\"a}t f{\"ur} Mathematik,}\\
\small{Thea-Leymann-Str.~9, 45127 Essen, Germany.}\\
}

\date{}

\begin{document}

\maketitle

\begin{abstract}
  In this paper we show the strong convergence of a 
  fully explicit space-time discrete approximation scheme for the solution process of the two-dimensional incompressible stochastic Navier-Stokes equations on the torus
  driven by additive 
  noise.
  To do so we apply an existing result 
  which was designed to prove strong convergence for the same approximation method for other stochastic partial differential equations with
  non-globally monotone non-linearities.
\end{abstract}


\section{Introduction}

In the last years some explicit and easily implementable versions of the explicit Euler method
have been proved to converge strongly (i.e.~in mean square) 
to the solutions of some infinite-dimensional stochastic evolution equations with superlinearly-growing non-linearities
either driven by trace class noise (e.g., \cite{gyongy2016convergence} and \cite{jentzen2015strong}) or by space-time white noise (e.g., \cite{becker2018strong} and \cite{hutzenthaler2016strong}).

The reasons to introduce versions of the Euler method rely on the fact that it was proved in, e.g., \cite[Theorem~2.1]{hutzenthaler2010strong} that in general the explicit and the linear-implicit Euler schemes do not converge strongly to the solutions of stochastic evolution equations with superlinarly-growing non-linearities. 
The difficulties for strong convergent drift-implicit Euler methods, instead, are related to the implementation: at each step a non-linear equation has to be solved approximately and consequently the computational cost increases with the dimension (see, e.g., \cite{hutzenthaler2012strong} for more details).

For the two-dimensional stochastic Navier-Stokes equations,
driven by additive or multiplicative noise, 
several existence and uniqueness results and several (strongly) convergent approximation schemes are available. 
The state of the art has been summarized very well in \cite{hausenblas2018time}. We refer the reader to this paper and also to \cite{bessaih2018strong},
where the authors establish rates of strong convergence for two approximation methods in the case of diffusion coefficients with linear growth: 
the fully implicit and also the semi implicit Euler schemes introduced in \cite{carelli2012rates} (also in the case of additive noise)
and the splitting scheme of \cite{bessaih2014splitting}.
Previously, except for \cite{dorsek2012semigroup}, who considered additive noise, there had been no result for the strong convergence rates of approximation schemes for the two-dimensional stochastic Navier-Stokes equations, only rates of convergence in probability were available.  

Let us now consider the full-discrete (both in space and in time) non-linearity-truncated accelerated exponential Euler-type scheme introduced in \cite{hutzenthaler2016strong}
which is the first strongly convergent approximation method for the solutions of stochastic Kuramoto-Sivashinsky equations driven by (a spatial distributional derivative of) space-time white noise.
Using a modified version of the scheme the strong convergence for stochastic Burgers equations and Allen-Cahn equations both driven by space-time white noise was proved in
\cite{jentzen2017strong}.
Moreover in \cite{becker2017strong} the spatial and temporal rates of convergence were established for space-time white noise driven Allen-Cahn equations.


In this document we show that the above mentioned numerical approximation provides an implementable scheme also for the solution of two-dimensional stochastic Navier-Stokes equations driven by some trace class noise:
\[ \begin{cases}
d X_t (x) = \left(\Delta X_t (x) - P ( \nabla X_t \cdot X_t)(x)\right) d t + B \, dW_t(x), & x\in (0,1)^2, t\in [0,T],
\\
X_0=\xi \in H,
\end{cases}
\]
with periodic boundary conditions and incompressibility condition $\operatorname{div} X_t=0$,
and where
$H$ is an appropriate (Hilbert) subspace of $L^{2}(\lambda_{(0,1)^2};\R^2)$ (with basis consisting of divergence free functions) in which $X_t$ for all $t\in [0,T]$ lives, 
$P$ is the projection on $H$, 
$W$ is an $\Id_H$-cylindrical Wiener process, 
and $B=(-\Delta)^{-\nicefrac12-\varepsilon}$, $\varepsilon \in (0,\infty)$, is a Hilbert-Schmidt operator.
For simplicity we have taken the viscosity coefficient $\nu$, that is one of the parameters for Navier-Stokes equations, equal to $1$.
Moreover for simplicity, we have taken the coefficient of the nonlinearity $c_1=1$ in Setting~\ref{sett:F} and $\kappa=\vt=0$ in Settings~\ref{sett:operator} and \ref{sett:F}, otherwise the drift would involve a linear term $c X$, for $c\in \R$.

Let the interpolation spaces $H_{r}$, $r\in \R$, associated to $(-\Delta)$. 
In particular for $r\in [0,\infty)$ it holds that $H_r$ is the domain of the fractional power 
$(-\Delta)^{r}$ of the operator $(-\Delta)$.
Let $\varepsilon\in (0,\infty)$, $\varrho \in (\nicefrac12, \nicefrac12+ \varepsilon)$, $\gamma \in (\varrho,\infty)$, and $\xi \in H_{\gamma}$.
Then we can consider the mild solution $X \colon [0,T] \times \Omega \to H_\varrho$
satisfying for all $t\in [0,T]$ that $\P$-a.s.
\begin{equation}\label{eq1}
	X_t = e^{t \Delta} \xi 
	+ \int_0^t e^{(t-s) \Delta} P( - \nabla X_s \cdot X_s ) \, d s
	+ \int_0^t  e^{(t-s) \Delta} (-\Delta)^{\nicefrac12+\varepsilon} \, d W_s.
\end{equation}
Note that any strong or weak solution is also a mild solution, the pathwise uniqueness of the the mild solution follows from a Gronwall-type argument and the fact, demonstrated in Lemma~\ref{lem:locLip}, that the nonlinearity is Lipschitz on bounded sets.

We will prove in Item~\eqref{item:SC:SC} in Theorem~\ref{th:SC} that the following adaptation of the approximation scheme of \cite{hutzenthaler2016strong} converges strongly to \eqref{eq1}.
Let $\mathcal{O}^n, \Y^n \colon [0,T]\times \Omega \to P_n(H)$ be the stochastic processes satisfying for all $n\in \N$, $t\in [0,T]$ that it holds $\P$-a.s.~that
\begin{equation*}
\begin{split}
& \mathcal{O}^{n}_{t} = \int_0^t P_n e^{(t-s) \Delta} (-\Delta)^{\nicefrac12+\varepsilon} \, d W_s + P_n e^{t \Delta} \xi\\
& \Y^n_t = \mathcal{O}^{n}_{t}
\\
&\quad + \int_0^t P_n e^{(t-s) \Delta} \, \one_{\left\{\|(-\Delta)^{\varrho} \Y^n_{\lf s \rf_{h_n}}\|_H + \|(-\Delta)^{\varrho} \mathcal{O}^n_{\lf s \rf_{h_n}}\|_H \leq h_n^{-\chi}\right\}} P ( - \nabla \Y^n_{\lf s \rf_{h_{n}}} \cdot \Y^n_{\lf s \rf_{h_{n}}}) \, ds,
\end{split}
\end{equation*}
where $\chi\in (0,\infty)$ an appropriate constant, 
$(h_m)_{m\in \N}$ is a positive sequence converging to 0, 
and $P_n$ are projections on increasing finite dimensional spaces $P_n(H)\subseteq H$ to be specified later in Setting~\ref{sett:space}.

The proof of the strong convergence 
\[\limsup_{n\to \infty} \sup_{s\in [0,T]} \|X_s -\Y^n_s\|_H=0
\] (Item~\eqref{item:SC:SC} in Theorem~\ref{th:SC}) 
is an application of Theorem~3.5 in \cite{jentzen2017strong} which improved the results in \cite{hutzenthaler2016strong} 
by considering a suitable generalized coercivity-type condition (in Lemma~\ref{coer:NS} below). 
The coefficients involved in the latter condition are functions that, composed with a suitable transformation (called $\mathbb{O}$) of the Ornstein-Uhlenbeck process 
(called $\mathcal{O}$),
satisfy exponential integrability properties 
(in this document given by Item~\eqref{item:prop:O:4} in Proposition~\ref{prop:prop:O}).

The implementation of the scheme is obtained just by taking for all $n\in \N$ the sequence 
$\Y^n_{(k+1)h_n}$ for $k\in (-1,\frac{T}{h_n}-1)\cap \N$.
This yields a fully explicit space-time discrete approximation scheme.
To the best of the author's knowledge, Theorem~\ref{th:SC} is the first strong convergence result for 
fully explicit space-time discrete approximation processes for  two-dimensional stochastic Navier-Stokes equations.

\subsection{Outline of the paper}
The main result is in Section~\ref{sec:SC}.
In the others sections the assumptions of the theorem are checked.

In Section~\ref{sec:prelim} we give the formal definition of the operators and the spaces involved, moreover some elementary results are proved. 
For example properties of the eigenvalues and eigenfunctions of the Laplace operator and properties of some interpolation spaces. 
Several of the estimates involved can also be found in \cite[Section~4]{jentzen2016exponential} where exponential integrability properties for an approximation scheme are provided in the setting of some two-dimensional stochastic Navier-Stokes equations with multiplicative trace class noise.

Section~\ref{sec:F} is dedicated to the nonlinear part of the drift (i.e. $- \nabla X \cdot X$), namely its formal definition, the generalized coercivity-type condition, and the local Lipschitzianity on bounded sets.

In Section~\ref{sec:noise} the random perturbation is introduced and the properties of the stochastic convolution process and its approximating sequence are studied. 
We will obtain in Lemma~\ref{lemma:conv:rate} that the strong convergence rate for the approximation of the noise is strictly smaller than $2 (\nicefrac12 +\varepsilon-\varrho)$.
Lemma~\ref{lem:coerc_finite} is auxiliary for Proposition~\ref{prop:prop:O} where the exponential integrability properties are given. Lemma~\ref{lem:exist:O} establishes the existence of a continuous version for the stochastic convolution processes.
The arguments in the proofs in this section are similar to those contained in the papers proving the convergence for other equations. 
Indeed they
are adaptations or follow the arguments of 
\cite[Lemma~5.5, Lemma~5.2, Proposition 5.6, Proposition~5.4]{jentzen2017strong} (for stochastic Burgers and Allen-Cahn equations)
and therefore of \cite[Lemma~5.9, Lemma~5.6, Corollary 5.10, Corollary~5.8]{hutzenthaler2016strong} (for stochastic Kuramoto-Sivashinsky equations).

\subsection{Notation}
Throughout this article the following notation is used. \\
Let $ \mathbb{N} = \{1, 2, 3, \ldots \}$ be the set of all natural numbers. \\
We denote by $ \lf \cdot \rf_h \colon \R \to \R$, $ h \in (0, \infty)$, the \emph{round-ground} functions which satisfy for all $t \in \R$, $h \in (0, \infty)$ that 
\begin{equation*}
\lf t \rf_h = \max( (-\infty, t] \cap \{0, h, -h, 2h, -2h, \ldots\} ).
\end{equation*}
Moreover for two sets $A$ and $B$ satisfying $A\subseteq B$
we denote by $\Id_A \colon A \to A$ the identity function on $A$, i.e.~the function which satisfies for all $ a \in A$ that $\Id_A(a)=a$,
and by $\mathbbm{1}_A^B \colon B \to \{0,1\}$ the indicator function which satisfies for all $ a \in A$ that $\mathbbm{1}_A^B(a)=1$ and for all $b\in B\setminus A$ that $\mathbbm{1}_A^B(b)=0$.\\
For two measurable spaces $( A, \mathcal{A})$ and $( B, \mathcal{B})$  we denote by 
$ \mathcal{M}(\mathcal{A}, \mathcal{B})$ the set of all $\mathcal{A} / \mathcal{B}$-measurable functions. 
For a topological space $ (X, \tau) $ 
we denote by $ \mathcal{B}(X) $ the Borel sigma-algebra of 
$ (X, \tau) $.
For a set $A \in \mathcal{B}(\R)$ we denote by $\lambda_A \colon \mathcal{B}(A) \to [0, \infty]$ the Lebesgue-Borel measure on $A$. \\
For a measure space $(\Omega, \F, \mu)$, a measurable space $(S, \mathcal{S})$, a set $R$, and a function $f \colon \Omega \to R$ we denote by $\left[ f \right]_{\mu, \mathcal{S}} $ the set given by 
\begin{equation*}
\begin{split}
& \left[f \right]_{\mu, \mathcal{S}} \\
& = \left\{ g \in \mathcal{M}(\F, \mathcal{S}) \colon \left( \exists \, A \in \F \colon \mu(A)=0 \,\, \text{and} \,\, \{ \omega \in \Omega \colon f(\omega) \neq g(\omega)\} \subseteq A \right) \right\}\!. 
\end{split}
\end{equation*}
For all $d\in \N$ we denote by $|\cdot|_d$ the Euclidean norm of $\R^d$.
For all $\alpha \in (0,\infty)$ and $p\in [1,\infty)$ 
let $W^{\alpha,p}((0,1)^2,\R^2)$ be the Sobolev-Slobodeckij spaces (see,e.g., \cite[Section~2.1.2]{runst1996sobolev}).
Let us recall that in particular for real numbers $p \in [1, \infty)$, $\theta \in (0,1)$ 
and a $\mathcal{B}((0,1)^2) /\mathcal{B}(\R^2)$-measurable function $v \colon (0,1)^2 \to \R^2$ we denote by $\|v\|_{\W^{\theta, p}((0,1)^2, \R^2)}$  the extended real number given by 
\begin{equation*}
	\|v\|_{\W^{\theta, p}((0,1)^2, \R^2)} 
	= \left[ \iint_{(0,1)^2} \! |v(x)|_2^p \, dx + \iint_{(0,1)^2} \iint_{(0,1)^2} \! \tfrac{|v(x)-v(y)|_2^p}{|x-y|_2^{2+ \theta p}} \, dx \, dy\right]^{\frac{1}{p}} \! \! . 
\end{equation*}
Let $\partial \colon W^{1,2}({(0,1)^2},\R^2)\mapsto L^2(\lambda_{(0,1)^2};\R^{2 \times 2})$ be
the function which satisfy for all smooth function with compact support $\phi\in C^{\infty}_{cpt}({(0,1)^2},\R^2)$, $v\in W^{1,2}({(0,1)^2},\R^2)$, $i\in \{1,2\}$, that 
\[
	\left\langle  \partial_i v,[\phi]_{\lambda_{(0,1)^2},\mathcal{B}(\R^2)}\right\rangle_{L^2(\lambda_{(0,1)^2};\R^2)}=- \left\langle  v,[\tfrac{\partial}{\partial x_i} \phi]_{\lambda_{(0,1)^2},\mathcal{B}(\R^2)}\right\rangle_{L^2(\lambda_{(0,1)^2};\R^2)}
\] 
and  $\partial v=(\partial_1 v,\partial_2 v)$.\\
Furthermore let $ (\und{\cdot}) \colon\ \{ [v]_{\lambda_{(0,1)^2}, \mathcal{B}(\R^2)} \in L^0(\lambda_{(0,1)^2}; \R^2) \colon v \in {C}( (0,1)^2, \R^2 ) \} \to {C}((0,1)^2, \R^2)
$  
be the function which satisfies for all $v \in {C}( (0,1)^2, \R^2 ) $  that
\[
\und{[v]_{\lambda_{(0,1)^2}, \mathcal{B}(\R^2)}}=v
.\]

%
%
%
%

\section{Properties of the state space of the solution} \label{sec:prelim}

\begin{sett} \label{sett:space}
Let
$( U, \left\langle   \cdot , \cdot \right\rangle _U, \left\| \cdot \right\|_U )$ be the separable Hilbert space
\[\Big(L^2(\lambda_{(0,1)^2}; \R^2), \left\langle  \cdot , \cdot \right\rangle_{L^2(\lambda_{(0,1)^2}; \R^2)}, \left\| \cdot \right\|_{L^2(\lambda_{(0,1)^2}; \R^2)} \Big).\]
For all $k\in \Z$ let $\varphi_k \in C((0,1),\R)$ be the function such that for all $x\in (0,1)$ it holds that
\[
	\varphi_k(x)
	= \one_{\{0\}}^{\Z}(k) + \one_{\N}^{\Z}(k) \sqrt{2}\cos(2 k \pi x) + \one_{\N}^{\Z}(-k) 	\sqrt{2}\sin(- 2 k \pi x),
\]
let the following elements $U$ 
\[
	e_{0,0,0}=\left[\left\{(1,0)\right\}_{(x,y)\in {(0,1)^2}}\right]_{\lambda_{(0,1)^2},\mathcal{B}(\R^2)}, 
	\qquad 
	e_{0,0,1}=\left[\left\{(0,1)\right\}_{(x,y)\in {(0,1)^2}}\right]_{\lambda_{(0,1)^2},\mathcal{B}(\R^2)},
\]
and for all $k,l\in \mathbb{Z}^2\setminus\{(0,0)\}$ the elements
\begin{equation*}
	e_{k,l,0}
	=\left[\left\{\left(\tfrac{l \varphi_k(x)\varphi_l(y)}{\sqrt{k^2+l^2}},\tfrac{k \varphi_{-k}(x)\varphi_{-l}(y)}{\sqrt{k^2+l^2}}\right)\right\}_{(x,y)\in {(0,1)^2}}\right]_{\lambda_{(0,1)^2},\mathcal{B}(\R^2)} \!\!.
\end{equation*}
Moreover 
let $H\subseteq U$ be the closed subvector space of $U$ with orthonormal basis $\mathbb{H}=\{e_{0,0,1}\} \cup \{e_{i,j,0} \colon i,j\in \Z\}$ 
and let,  for all $n\in \N$, 
\[
	\mathbb{H}_n = \{e_{0,0,1}\}\cup \{e_{k,l,0} \colon k,l\in \Z \text{ and } k^2+l^2 < n^2 \} \subseteq \mathbb{H}
\] 
and $P_n \subseteq L(H)$ the projection on the finite dimensional subspace of $H$ spanned by $\mathbb{H}_n$, i.e.~ 
for all $ u \in H $ it holds that
$ 
P_n(u) = \sum_{h\in \mathbb{H}_n} \left\langle  h, u \right\rangle_H \, h.
$
In addition let
$ \kt \in (0,\infty)$ and $\lambda_{e_{0,0,1}}, \lambda_{e_{k,l,0}}\in [0,\infty)$, $k,l\in \Z$, be the following real numbers
$\lambda_{e_{0,0,1}}=\lambda_{e_{0,0,0}}=\kt$,
$\lambda_{e_{k,l,0}}=\kt+ 4 \pi^2(k^2+l^2)$.
\end{sett}

\subsection{Elementary estimates}



\begin{lemma}\label{lem:useful}
	Assume Setting~\ref{sett:space}.
	Then it holds
	\begin{enumerate}[(i)]
		\item \label{item:useful:3}
		for all $\varepsilon \in (0,\infty)$ that $ \sum_{h \in \mathbb{H}} \lambda_{h}^{-1-\varepsilon} <\infty$,
		\item \label{item:useful:2}
		for all $\beta \in (0,\infty)$, $\varepsilon\in [0,\beta)$ that 
		$ 
			\sum_{h \in \mathbb{H}} (\kappa+\lambda_h)^{\varepsilon} \lambda_{h}^{-1-\beta} <\infty
		$,
		\item \label{item:useful:1}
		for all $n\in \N$, $\varepsilon\in (0,\infty)$ that $\|(\kappa-A)^{-\varepsilon} (\Id_H-P_n) \|_{L(H)} \leq (\kappa + \kt+  4 \pi^2 n^2)^{-\varepsilon} $,
		\item  \label{item:useful:4}
		that $\liminf_{ n \to \infty } \inf( \{\lambda_h \colon h \in \mathbb{H} \backslash \mathbb{H}_n \} \cup \{\infty\}  ) = \infty$.
	\end{enumerate}
\end{lemma}
\begin{proof}[Proof of Lemma \ref{lem:useful}.]
	Throughout the proof of the first item, let $\varepsilon \in (0,\infty)$ be a fixed real number. Then note that
	\begin{equation} \label{eq:usefull:3:k}
	\begin{split}
	& \tfrac12 \smallsum_{k\in \mathbb{Z}\setminus\{0\}} \lambda_{e_{k,0,0}}^{-1-\varepsilon}  = \smallsum_{k\in \N} \lambda_{e_{k,0,0}}^{-1-\varepsilon}  
	= \smallsum_{k\in \N} \lambda_{e_{0,k,0}}^{-1-\varepsilon} 
	= \smallsum_{k\in \N} (\kt + 4 \pi^2 k^2)^{-1-\varepsilon} \\
	& \leq (2 \pi)^{-2(1+\varepsilon)} \smallsum_{k\in \N} k^{-2(1+\varepsilon)} 
	= (2 \pi)^{-2(1+\varepsilon)} \left(1+\smallsum_{k\in \N} (k+1)^{-2(1+\varepsilon)} \right)\\
	& \leq (2 \pi)^{-2(1+\varepsilon)}  \left(1+ \int_1^{\infty}  x^{-2(1+\varepsilon)} d x \right) <\infty
	\end{split}
	\end{equation} 
	and 
	\begin{equation*}
	\begin{split}
	& \smallsum_{l,k\in\N\setminus\{1\}} (\kt + 4 \pi^2 (k^2+l^2) )^{-(1+\varepsilon)}
	\leq 2 \pi \int_1^\infty  x \left(\kt + 4 \pi^2 x^2 \right)^{-(1+\varepsilon)}  d x \\
	& =  \int_{\sqrt{ \kt +  4 \pi^2 }}^\infty y^{1-2(1+\varepsilon)} d y = \frac1{2\beta} \left(\kt+  4 \pi^2 \right)^{-\varepsilon} <\infty.
	\end{split}
	\end{equation*}
	This, together with \eqref{eq:usefull:3:k}, implies 
	\begin{equation} \label{eq:usefull:3:kl}
	\begin{split}
	& \smallsum_{k,l \in\Z\setminus\{0\}}  |\lambda_{e_{k,l,0}} |^{-(1+\varepsilon)}  \\
	& =\smallsum_{k,l\in\Z\setminus\{0\}} (\kt + 4 \pi^2 (k^2+l^2) )^{-(1+\varepsilon)} \\
	& = 4 \smallsum_{l,k\in\N} (\kt + 4 \pi^2 (k^2+l^2) )^{-(1+\varepsilon)}  \\
	& = 8 \smallsum_{k\in\N} (\kt + 4 \pi^2 + 4 \pi^2 k^2)^{-(1+\varepsilon)} + 4 \smallsum_{l,k\in\N\setminus\{1\}} (\kt + 4 \pi^2 (k^2+l^2) )^{-(1+\varepsilon)}\\ 
	& \leq 4 \smallsum_{k\in\Z\setminus\{0\}} \lambda_{e_{k,0,0}}^{-(1+\varepsilon)} + 4 \smallsum_{l,k\in\N\setminus\{1\}} (\kt + 4 \pi^2 (k^2+l^2) )^{-(1+\varepsilon)}<\infty.
	\end{split}
	\end{equation}
	Combining \eqref{eq:usefull:3:k} and \eqref{eq:usefull:3:kl} with the fact that
	\begin{equation*}
	\begin{split}
	& \smallsum_{h\in \mathbb{H}} \lambda_h^{-(1+\varepsilon)} 
	= \lambda_{e_{0,0,0}}^{-(1+\varepsilon)} + \lambda_{e_{0,0,1}}^{-(1+\varepsilon)} + \smallsum_{(k,l)\in \mathbb{Z}^2\setminus\{(0,0)\}} \lambda_{e_{k,l,0}}^{-(1+\varepsilon)}\\
	& = 2 \kt^{-(1+\varepsilon)} + \smallsum_{k \in \mathbb{Z}\setminus \{0\}} \lambda_{e_{k,0,0}}^{-(1+\varepsilon)} + \smallsum_{k\in \mathbb{Z}\setminus\{0\}} \lambda_{e_{0,k,0}}^{-(1+\varepsilon)} + \smallsum_{k,l\in \mathbb{Z}\setminus\{0\}} \lambda_{e_{k,l,0}}^{-(1+\varepsilon)}
	\end{split}
	\end{equation*}
	proves Item~\eqref{item:useful:3}.

	In the proof of Item~\eqref{item:useful:2} 
	let $\beta \in (0,\infty)$ and $\varepsilon \in [0,\beta)$ be fixed real numbers.
	Then note that there exists $m\in \N$ such that for all $h\in \mathbb{H}\setminus \mathbb{H}_m$ it holds that $\kappa \leq \lambda_h$. 
	This implies that
	\begin{equation*}
	\begin{split}
	\smallsum_{h \in \mathbb{H}} (\kappa+\lambda_h)^{\varepsilon} \lambda_{h}^{-1-\beta} 
	& = \smallsum_{h \in \mathbb{H}_m}  (\kappa+\lambda_h)^{\varepsilon} \lambda_{h}^{-1-\beta} + \smallsum_{h \in \mathbb{H}\setminus \mathbb{H}_m} (\kappa+\lambda_h)^{\varepsilon} \lambda_{h}^{-1-\beta}\\
	& \leq 
	\smallsum_{h \in \mathbb{H}_m}  (\kappa+\kt + 4 \pi^2 |h|^2 )^{\varepsilon} (\kt)^{-1-\beta} 
	+ 2^{\varepsilon} \smallsum_{h \in \mathbb{H}\setminus \mathbb{H}_m} \lambda_h^{\varepsilon-1-\beta}.
	\end{split}
	\end{equation*}
	This, 
	the fact that $\#_{\mathbb{H}_m}<\infty$, 
	the fact that $\lambda_h > 0$ for all $h\in \mathbb{H}$,
	and Item~\eqref{item:useful:3} (with $\varepsilon= \beta-\varepsilon$) demonstrate Item~\eqref{item:useful:2}.

	Throughout the proof of Item~\eqref{item:useful:1} let the real number $\varepsilon\in (0,\infty)$ and the natural number $n\in \N$ be fixed. 
	Then observe that for all $h\in \mathbb{H}_n$ it holds that $ (\Id_H-P_n) h = 0$ and for all $h\in \mathbb{H}\setminus \mathbb{H}_n$ it holds that $(\Id_H-P_n) h = h$.
	This, together with the fact that $v\in H$ that $v=\sum_{h\in \mathbb{H}} \left\langle  v,h \right\rangle_H h$, shows that it holds for all $v\in H$ that
	\begin{equation*}
	\begin{split}
	& \|(\kappa-A)^{-\varepsilon} (\Id_H-P_n) v\|^2_H = \left\| \smallsum_{h\in \mathbb{H}\setminus \mathbb{H}_n} \left\langle  v,h \right\rangle_H (\kappa-A)^{-\varepsilon} h \right\|_{H}^2\\
	& = \left\| \smallsum_{h\in \mathbb{H}\setminus \mathbb{H}_n} (\kappa+ \lambda_h)^{-\varepsilon} \left\langle  v,h \right\rangle_H  h \right\|_{H}^2\\
	& = \smallsum_{h\in \mathbb{H}\setminus \mathbb{H}_n} (\kappa+ \lambda_h)^{-2\varepsilon} \left\langle  v,h \right\rangle_H^2.
	\end{split}
	\end{equation*}
	This, together with the fact that for all $h\in \mathbb{H}\setminus \mathbb{H}_n$ it holds that $\lambda_h \geq \kt + 4 \pi^2 n^2$, shows that it holds for all $v\in H$ that 
	\begin{equation*}
	\begin{split}
	& \|(\kappa-A)^{-\varepsilon} (\Id_H-P_n) v\|^2_H 
	\leq \smallsum_{h\in \mathbb{H}\setminus \mathbb{H}_n} (\kappa+ \kt + 4 \pi^2 n^2)^{-2\varepsilon} \left\langle  v,h \right\rangle_H^2\\
	& \leq (\kappa + \kt + 4 \pi^2 n^2)^{-2\varepsilon} \smallsum_{h\in \mathbb{H}}  \left\langle  v,h \right\rangle_H^2
	= (\kappa + \kt + 4 \pi^2 n^2)^{-2\varepsilon} \|v\|^2_H .
	\end{split}
	\end{equation*}
	Therefore, we obtain that 
	\begin{equation*}
	\begin{split}
	& \|(\kappa-A)^{-\varepsilon} (\Id_H-P_n) \|_{L(H)} \\
	& = \sup\left\{ \|(\kappa-A)^{-\varepsilon} (\Id_H-P_n) v\|_H \colon v\in H \text{ with } \|v\|_H=1\right\}\\
	& \leq (\kappa + \kt + 4 \pi^2 n^2)^{-\varepsilon}  .
	\end{split}
	\end{equation*}
	This establishes Item~\eqref{item:useful:1}.

	Finally note that it holds for all $n\in \N$ that $\inf\{\lambda_h \colon h\in \mathbb{H}\setminus \mathbb{H}_n\} = \lambda_{e_{n,0,0}} = \kt + 4 \pi^2 n^2$.
	This proves Item~\eqref{item:useful:4}.
	The proof of Lemma~\ref{lem:useful} is thus completed.
\end{proof}

\begin{lemma}\label{lem:eigenf}
	Assume Setting~\ref{sett:space}.
	Then it holds
	\begin{enumerate}[(i)]
		\item \label{item:eigenf:L4}
		that $\sup_{h \in \mathbb{H}} \|h\|_{L^{\infty}(\lambda_{(0,1)^2}; \R^2)} \leq 2$,
		\item \label{item:eigenf:7}
		for all $h=(h_1,h_2) \in \mathbb{H}$ that 
		$\partial_1 h_1 +\partial_2 h_2 =\left[\{0\}_{x\in {(0,1)^2}}\right]_{\lambda_{(0,1)^2},\mathcal{B}(\mathbb{R})}$,
		\item  \label{item:eigenf:6}
		for all $j\in\{1,2\}$, $h,v \in \mathbb{H}$ with $v\neq h$  
		that $\left\langle  \partial_j h, \partial_j v\right\rangle_H = 0$,
		and
		\item \label{item:eigenf:5}
		for all $r\in [\nicefrac12 ,\infty)$ that $\max_{j\in \{1,2\}} \sup_{h\in \mathbb{H}} {\|\partial_j h\|_U}{|\lambda_h|^{-r}} \leq 1$.
	\end{enumerate}
\end{lemma}
\begin{proof}[Proof of Lemma \ref{lem:eigenf}.]
	First note that for all $h\in \mathbb{H}$ it holds that $\|h\|_{L^{\infty}(\lambda_{(0,1)^2}; \R^2)}= \sup_{x\in (0,1)^2}|\und{h}(x)|_2$. 
	In particular it holds that $\|e_{0,0,0}\|_{L^{\infty}(\lambda_{(0,1)^2}; \R^2)}=\|e_{0,0,1}\|_{L^{\infty}(\lambda_{(0,1)^2}; \R^2)}=1$ and for all $(k,l)\in \Z^2\setminus\{(0,0)\}$ it holds that 
	\begin{equation*}
	\begin{split}
	\|e_{k,l,0}\|_{L^{\infty}(\lambda_{(0,1)^2}; \R^2)}
	&= 
	\sup_{x,y\in (0,1)} \left( \tfrac1{\sqrt{k^2+l^2}} \left| l  \varphi_k(x)\varphi_l(y), k \varphi_{-k}(x)\varphi_{-l}(y) \right|_2\right)\\
	& =
	\sup_{x,y\in (0,1)}  \left(\tfrac1{\sqrt{k^2+l^2}} ( l^2  (\varphi_k(x)\varphi_l(y))^2 + k^2 (\varphi_{-k}(x)\varphi_{-l}(y))^2 )^{\nicefrac12} \right)\\
	& \leq 
	\tfrac1{\sqrt{k^2+l^2}}  \left( 2^2 (l^2 + k^2 )\right)^{\nicefrac12}=2.
	\end{split}
	\end{equation*}
	This establishes Item~\eqref{item:eigenf:L4}.

	Note that for all $j\in \{1,2\}$, $n_1,n_2\in \Z$ it holds that 
	\begin{equation} \label{eq:eigenf:deltah}
	\begin{split}
	\partial_{j} (e_{n_1,n_2,0}) = 2 \pi (-1)^j n_j  \, e_{(-1)^j n_1, (-1)^{j+1} n_2,0}, \quad \partial_j e_{0,0,1} = 0.
	\end{split}
	\end{equation}
	%
	%
	%
	This implies that for all $n_1,n_2 \in \Z$ it holds that $e_{n_1,n_2,0} = \left( (e_{n_1,n_2,0})_1, (e_{n_1,n_2,0})_2 \right)$ and
	\begin{equation*}
	\begin{split}
	& \partial_1 (e_{n_1,n_2,0})_1 +\partial_2 (e_{n_1,n_2,0})_2 
	\\	
	&=
	- 2 \pi n_1  (e_{-n_1,n_2,0})_1  + 2 \pi n_2 (e_{n_1,-n_2,0})_1 
	\\
	&=
	\left[\left\{\tfrac{- 2 \pi n_1  n_2 \varphi_{-n_1}(x)\varphi_{n_2}(y)  + 2 \pi n_2 n_1 \varphi_{-n_1}(x)\varphi_{n_2}(y)  }{\sqrt{n_1^2+n_2^2}}\right\}_{(x,y)\in {(0,1)^2}}\right]_{\lambda_{(0,1)^2},\mathcal{B}(\mathbb{R})}
	\\
	&= 
	\left[\{0\}_{(x,y)\in {(0,1)^2}}\right]_{\lambda_{(0,1)^2},\mathcal{B}(\mathbb{R})}.
	\end{split}
	\end{equation*}
	This and \eqref{eq:eigenf:deltah} demonstrate Item~\eqref{item:eigenf:7}.

	The fact that $\mathbb{H}$ is an orthonormal basis together with \eqref{eq:eigenf:deltah} establishes for all $j\in \{1,2\}$, $n_1,n_2, m_1,m_2 \in \Z$ with $n_1 \neq m_1$ or $n_2 \neq m_2$ that
	\begin{equation*}
	\begin{split}
	& \left\langle  \partial_j e_{n_1,n_2,0}, \partial_j e_{m_1,m_2,0}\right\rangle_H 
	\\
	& 
	= 4 \pi^2 n_j m_j \left\langle  e_{(-1)^j n_1, (-1)^{j+1} n_2,0}, e_{(-1)^{j} m_1, (-1)^{j+1} m_2,0}\right\rangle_H
	=0,
	\end{split}
	\end{equation*}
	and $ \left\langle  \partial_j e_{n_1,n_2,0}, \partial_j e_{0,0,1}\right\rangle_H =0$.
	This demonstrates Item~\eqref{item:eigenf:6}.

	The fact that for all $h\in \mathbb{H}$ it holds that $\|h\|_U=1$ shows for all $r\in \R$ that 
	\begin{equation*}
	\begin{split}
	& \max_{j\in \{1,2\}} \sup_{h\in \mathbb{H}} \tfrac{\|\partial_j h\|_U}{|\lambda_h|^{r}} 
	= \max_{j\in \{1,2\}} \sup_{(n_1,n_2)\in \Z\setminus\{(0,0)\}} \tfrac{\|\partial_j e_{n_1, n_2, 0}\|_U}{|\lambda_{e_{n_1, n_2, 0}}|^{r}}
	\\
	& =   
	\max_{j\in \{1,2\}} \sup_{n_1,n_2\in \Z\setminus\{0\}} 
	\tfrac{ 2 \pi |n_j| \|  e_{(-1)^j n_1, (-1)^{j+1} n_2,0}\|_U}{|\lambda_{e_{n_1, n_2, 0}}|^{r}}
	\\
	&  =   
	\max_{j\in \{1,2\}} \sup_{n_1,n_2\in \Z\setminus\{0\}} 
	\tfrac{ 2 \pi |n_j|}{|\lambda_{e_{n_1, n_2, 0}}|^{r}}.
	\end{split}
	\end{equation*}
	The fact that for all $j\in \{1,2\}$, $n_1,n_2\in \N\setminus \{0\}$ it holds that $1 \leq n_j \leq \sqrt{n_1^2+n_2^2}$ 
	implies for all $j \in \{1,2\}$, $n_1,n_2\in \Z\setminus \{0\}$, $r\in [\nicefrac12,\infty)$ that 
	$ 1 \leq 2 \pi |n_j|  \leq  |\lambda_{e_{n_1, n_2, 0}}|^{\nicefrac12} \leq |\lambda_{e_{n_1, n_2, 0}}|^{r}$. 
	Hence for all $r\in [\nicefrac12,\infty)$ it holds that
	\begin{equation*}
	\begin{split}
	& \max_{j\in \{1,2\}} \sup_{h\in \mathbb{H}} \tfrac{\|\partial_j h\|_U}{|\lambda_h|^{r}} 
	\leq
	\max_{j\in \{1,2\}} \sup_{(n_1,n_2)\in \Z\setminus \{0\}} 
	\tfrac{ 2 \pi |n_j| }{|\lambda_{e_{n_1, n_2, 0}}|^{\nicefrac12}} 
	\leq 1.
	\end{split}
	\end{equation*}
	This establishes Item~\eqref{item:eigenf:5}.
	The proof of Lemma~\ref{lem:eigenf} is thus completed.
\end{proof}

\subsection{Properties of the spaces involved}

\begin{sett}(The Laplace operator with periodic boundary conditions) \label{sett:operator}
Assume Setting~\ref{sett:space},
let
$ A \colon D(A) \subseteq H \to H $
be the linear operator which satisfies
$ D(A) = \{ v \in H \colon \sum_{h\in \mathbb{H}} | \lambda_h \left\langle  h , v \right\rangle_H |^2 < \infty \} $
and
$ \forall \, v \in D(A) \colon A v = \sum_{h\in \mathbb{H}} - \lambda_h \left\langle  h , v \right\rangle_H \, h$,
let $\kappa\in [0,\infty)$,
and
let
$ ( H_r, \left\langle   \cdot , \cdot \right\rangle _{ H_r }, \left\| \cdot \right\|_{ H_r } ) $, $ r \in \R $,
be a family of interpolation spaces associated to $ \kappa - A $ (see, e.g., \cite[Section~3.7]{sell2013dynamics}).
\end{sett}

\begin{lemma}[Integration by parts]
	\label{lem:NS:ibp}
	Assume Setting \ref{sett:operator} and let $r \in [\nicefrac12,\infty)$, $\zeta \in (\nicefrac12,\infty)$.
	Then it holds
	\begin{enumerate}[(i)]
		\item \label{eq:NS:ibp:gradient_norm} 
		for all $v\in H_r$,  $j\in \{1,2\}$ that $v\in W^{1,2}({(0,1)^2}, \R^2)$, $\partial_j v =\sum_{h\in \mathbb{H}} \left\langle  h,v\right\rangle _H \partial_j h$,
		and 
		$\|\partial_j v\|_U \leq \|v\|_{H_{r}}$,
		\item \label{eq:NS:ibp:0divergence} 
		for all $v=(v_1,v_2)\in H_r$ that $\partial_1 v_1 +\partial_2 v_2 =\left[\{0\}_{x\in {(0,1)^2}}\right]_{\lambda_{(0,1)^2},\mathcal{B}(\mathbb{R})}$,
		\item \label{eq:NS:Hzeta}
		$H_\zeta \subseteq L^{\infty}(\lambda_{(0,1)^2}; \R^2)$,
		\item \label{eq:NS:ibp} 
		for all $i,j,k,l\in \{1,2\}$, $u,v,w\colon (0,1)^2 \to \R^2$ satisfying 
		\[
			[u]_{\lambda_{(0,1)^2},\mathcal{B}(\R^2)}, [v]_{\lambda_{(0,1)^2},\mathcal{B}(\R^2)}, [w]_{\lambda_{(0,1)^2},\mathcal{B}(\R^2)}\in H_\zeta
		\]
		that $[v_i \cdot w_j]_{\lambda_{(0,1)^2},\mathcal{B}(\R)} \in W^{1,2}({(0,1)^2},\R) \cap L^{\infty}(\lambda_{(0,1)^2};\R)$, 
		\begin{equation}
		\partial_k( v_i w_j )= \partial_k v_i w_j+v_i \partial_k w_j,
		\end{equation} 
		and 
		$\left\langle  \partial_k (v_i w_j),u_l\right\rangle _{L^2(\lambda_{(0,1)^2};\R)}=-\left\langle  (v_i w_j),\partial_k u_l\right\rangle _{L^2(\lambda_{(0,1)^2};\R)}$,
		\item \label{eq:NS:ibp:FinU}
		for all $v,w=(w_1,w_2)\in H_{\zeta}$ that
		$\sum_j w_j \partial_j v \in U$.
	\end{enumerate}
\end{lemma}
\begin{proof}[Proof of Lemma \ref{lem:NS:ibp}.]
	Let us first observe that combining Item~\eqref{item:eigenf:5} in Lemma~\ref{lem:eigenf} and Item (i) and Item (ii) Lemma~4.4 in \cite{jentzen2016exponential}(with $\rho=r, u=v, j=j$ for $v\in H_{r}$, $j\in \{1,2\}$)
	proves that it holds for all $v\in H_{r}$, $j\in \{1,2\}$ that
	$H_{r}\subseteq W^{1,2}({(0,1)^2},\R^2)$, 
	$\partial_j v =\sum_{h\in \mathbb{H}} \left\langle  h,v\right\rangle _H \partial_j h$,
	and $\|\partial_j v\|_U \leq \left(\sup_{h\in H} \|\partial_j h\|_U|\lambda_h|^{-r}  \right) \|v\|_{H_{r}}$. This and Item~\eqref{item:eigenf:5} in Lemma~\ref{lem:eigenf} ensure Item~\eqref{eq:NS:ibp:gradient_norm}.

	Moreover  the fact that for all $v\in H_{r}$, $j\in \{1,2\}$ it holds that
	$\partial_j v =\sum_{h\in \mathbb{H}} \left\langle  h,v\right\rangle _H \partial_j h$ implies that for all $v\in H_{r}, j\in \{1,2\}$ it holds that
	$\|\partial_j v_j -\sum_{h\in \mathbb{H}} \left\langle  h,v\right\rangle _H \partial_j h_j\|_{L^2(\lambda_{(0,1)^2};\R)} 
	\leq \|\partial_j v -\sum_{h\in \mathbb{H}} \left\langle  h,v\right\rangle _H \partial_j h\|_{U}=0$.
	This together with Item~\eqref{item:eigenf:7} in Lemma~\ref{lem:eigenf} shows for all $v=(v_1,v_2)\in H_{r}$ that
	\begin{equation*}
	\begin{split}
	& \|\smallsum_{j=1}^2 \partial_j v_j \|_{L^2(\lambda_{(0,1)^2};\R)} 
	\\
	& =
	\|\smallsum_{j=1}^2 \partial_j v_j -  \smallsum_{h=(h_1,h_2)\in \mathbb{H}} \left\langle  h, v \right\rangle_H \smallsum_{j=1}^2 \partial_j h_j \|_{L^2(\lambda_{(0,1)^2};\R)}
	\\
	& \leq
	\smallsum_{j=1}^2   \|\partial_j v_j -  \smallsum_{h=(h_1,h_2)\in \mathbb{H}} \left\langle  h, v \right\rangle_H  \partial_j h_j \|_{L^2(\lambda_{(0,1)^2};\R)}
	=0.
	\end{split}
	\end{equation*}
	This establishes Item~\eqref{eq:NS:ibp:0divergence}.

	Next note that 
	$\sum_{h\in \mathbb{H}}|\lambda_h|^{-2\zeta}<\infty$ (see e.g. Item~\eqref{item:useful:3} in Lemma~\ref{lem:useful}).
	Combining this with Items~\eqref{item:eigenf:L4} and \eqref{item:eigenf:5} in Lemma~\ref{lem:eigenf} with
	Lemma~4.9 in \cite{jentzen2016exponential} (with $\rho=\zeta$, $u=v$) 
	and \cite[Lemma~4.5 and Lemma~4.7]{jentzen2016exponential}
	establishes Item~\eqref{eq:NS:Hzeta} and Item~\eqref{eq:NS:ibp}.

	Let $v,w\in H_\zeta$ be fixed for the entire proof of Item~\eqref{eq:NS:ibp:FinU}. 
	Note that Item~\eqref{eq:NS:Hzeta} and Cauchy-Schwarz inequality ensure that $w \in L^{\infty}(\lambda_{(0,1)^2};\R^2)$ and $\sum_{j=1}^2 \|w_j\|_{L^{\infty}(\lambda_{(0,1)^2};\R)} \leq \sqrt{2} \|w\|_{L^{\infty}(\lambda_{(0,1)^2};\R^2)}<\infty$.
	Moreover Item~\eqref{eq:NS:ibp:gradient_norm} assures that $\|\partial_j v\|_U \leq \|v\|_{H_{\zeta}}<\infty$ for all $j\in \{1,2\}$.
	This and the triangle inequality show that
\begin{equation*}
	\big\| \smallsum_{j=1}^2 w_j \partial_j v \big\|_U 
	\leq \smallsum_{j=1}^2 \|  w_j \|_{L^{\infty}(\lambda_{(0,1)^2};\R)} \|\partial_j v \|_U
	\leq \sqrt{2} \|v\|_{H_\zeta} \|  w \|_{L^{\infty}(\lambda_{(0,1)^2};\R^2)} 
	<\infty.
\end{equation*}
	This establishes Item~\eqref{eq:NS:ibp:FinU}.
	
	%
	%
	%
	%
	The proof of Lemma~\ref{lem:NS:ibp} is thus completed.
\end{proof}


\begin{lemma}[Sobolev embeddings] \label{lem:NS:cont} 
	Assume Setting \ref{sett:operator} and let $\zeta\in (\nicefrac12,\infty)$, $v \in H_{\zeta}$, $\beta\in (0,1)$, $p\in (\nicefrac{2}{\beta},\infty)$, $w\in \W^{\beta, p}((0,1)^2, \mathbb{R}^2) $.
	Then there exist $u_1,u_2 \in C({(0,1)^2},\R^2)$ such that $v=[u_1]_{\lambda_{(0,1)^2}, \mathcal{B}(\R^2)}$ and $w=[u_2]_{\lambda_{(0,1)^2}, \mathcal{B}(\R^2)}$. 
\end{lemma}
\begin{proof}[Proof of Lemma \ref{lem:NS:cont}.]
	First, note that $v\in H_\zeta \subseteq W^{2\zeta,2}({(0,1)^2},\R^2)$ hence Sobolev embedding theorem proves that there exists $u_1\in \mathcal{C}({(0,1)^2},\R^2)$ such that $u=[u_1]_{\lambda_{(0,1)^2}, \mathcal{B}(\R^2)}$.
	Sobolev embedding theorem ensures that there exists $u_2\in \mathcal{C}^{0,\frac{\beta p - 2}{p}}({(0,1)^2},\R^2)$ such that $w=[u_2]_{\lambda_{(0,1)^2}, \mathcal{B}(\R^2)}$.
	The proof of Lemma~\ref{lem:NS:cont} is thus completed.
\end{proof}

\section{Properties of the non linearity} \label{sec:F}

\begin{sett}\label{sett:F}
Assume Setting~\ref{sett:operator},
$ c_1, \vt \in \R$,
$\rho\in (\nicefrac12,1)$,
let $R\in L(U)$ be the orthogonal projection of $U$ on $H$,
and let $F \colon H_{\rho} \to H$ satisfy for all 
$v \in H_{\rho} $ that 
\begin{equation}\label{setting:F:B}
F(v)=R\big(\vt \, v -c_1\smallsum_{i=1}^2 v_i \, \partial_i v\big).
\end{equation}
\end{sett}

\noindent Note that Item~\eqref{eq:NS:ibp:FinU} in Lemma~\ref{lem:NS:ibp} assures that the function in \eqref{setting:F:B} is well defined.

\begin{lemma}[Generalized coercivity-type condition]
	\label{coer:NS}
	Assume Setting \ref{sett:F},
	let $\varepsilon\in (0,\infty)$ and let $v=(v_1,v_2), w=(w_1,w_2) \in H_{\rho}$.
	Then it holds that
	\begin{equation*}
	\begin{split}
	\left| \left\langle   v, F( v + w ) \right\rangle _H \right| 
	& \leq 
	\left( \frac{3}{2}|\vt| + \frac{ c_1^2}{2\varepsilon} \big[\!\sup\nolimits_{x \in (0,1)^2} |\und{w}(x)|_2^2\big]\right) \| v \|^2_H 
	+ 2 \varepsilon \|v\|^2_{H_{\nicefrac12}}
	\\
	& \quad + 
	\left(\frac{|\vt|}{2} \big[\!\sup\nolimits_{x \in (0,1)^2} |\und{w}(x)|_2^2\big] + \frac{ c_1^2}{2\varepsilon}\big[\!\sup\nolimits_{x \in (0,1)^2} |\und{w}(x)|_2^4\big]\right) 
	.
	\end{split}
	\end{equation*}
\end{lemma}
\begin{proof}[Proof of Lemma \ref{coer:NS}.]
	First note that 
	\begin{equation*}
	\begin{split}
	& \left\langle  v, F( v + w ) \right\rangle_H  = \left\langle  v, F( v + w ) \right\rangle_U \\
	& =  \vt \left\langle   v, v+w \right\rangle _H   -c_1 \smallsum_{j,i=1}^2 \left\langle   v_i,  (v_j+w_j) \partial_j (v_i+w_i) \right\rangle _{L^2(\lambda_{(0,1)^2}; \R)}.
	\end{split}
	\end{equation*}
	This and Item~\eqref{eq:NS:ibp} in Lemma~\ref{lem:NS:ibp} yield
	\begin{equation*}
	\begin{split}
	& \left\langle   v, F( v + w ) \right\rangle _H  \\
	& =  \vt \left\langle   v, v+w \right\rangle _H   + c_1 \smallsum_{j,i=1}^2 \left\langle   \partial_j \left(v_i ( v_j +w_j)\right),  v_i + w_i\right\rangle _{L^2(\lambda_{(0,1)^2}; \R)}\\
	& =  \vt \left\langle   v, v+w \right\rangle _H   + c_1 \smallsum_{j,i=1}^2 \left\langle   \partial_j v_i , ( v_j +w_j) ( v_i + w_i)\right\rangle _{L^2(\lambda_{(0,1)^2}; \R)} \\
	& \quad + c_1 \left\langle    \smallsum_{j=1}^2\partial_j  v_j +  \smallsum_{j=1}^2 \partial_j w_j,  \smallsum_{i=1}^2 ( v^2_i + v_i w_i) \right\rangle _{L^2(\lambda_{(0,1)^2}; \R)}.
	\end{split}
	\end{equation*}
Item~\eqref{eq:NS:ibp:0divergence} in Lemma~\ref{lem:NS:ibp} hence shows that 
	\begin{equation*}
	\begin{split}
	& \left\langle   v, F( v + w ) \right\rangle _H  \\
	& =  \vt \left\langle   v, v+w \right\rangle _H   + c_1 \smallsum_{j,i=1}^2 \left\langle   \partial_j v_i , ( v_j +w_j) ( v_i + w_i)\right\rangle _{L^2(\lambda_{(0,1)^2}; \R)}.
	\end{split}
	\end{equation*}
Moreover observe that for all $u\in H_{\rho}$ it holds that 
	$
	\sum_{j,i=1}^2 \left\langle    \partial_j v_i, u_j v_i\right\rangle _{L^2(\lambda_{(0,1)^2}; \R)}=0
	$ 
because Item~\eqref{eq:NS:Hzeta}, Item~\eqref{eq:NS:ibp}, and Item~\eqref{eq:NS:ibp:0divergence} in Lemma~\ref{lem:NS:ibp} imply that
	\begin{equation*}
	\begin{split}
	& \smallsum_{j,i=1}^2 \left\langle   u_j \partial_j v_i, v_i\right\rangle _{L^2(\lambda_{(0,1)^2}; \R)} =  - \smallsum_{j,i=1}^2 \left\langle   v_i,  \partial_j (u_j v_i)\right\rangle _{L^2(\lambda_{(0,1)^2}; \R)}\\
	& = - \smallsum_{j,i=1}^2 \left\langle   v^2_i,  \partial_j u_j\right\rangle _{L^2(\lambda_{(0,1)^2}; \R)} - \smallsum_{j,i=1}^2 \left\langle   u_j \partial_j v_i, v_i \right\rangle _{L^2(\lambda_{(0,1)^2}; \R)}\\
	& = -  \left\langle   \smallsum_{i=1}^2 v^2_i, \smallsum_{j=1}^2 \partial_j u_j\right\rangle _{L^2(\lambda_{(0,1)^2}; \R)} - \smallsum_{j,i=1}^2 \left\langle   u_j \partial_j v_i, v_i \right\rangle _{L^2(\lambda_{(0,1)^2}; \R)}\\
	& =
	- \smallsum_{j,i=1}^2 \left\langle   u_j \partial_j v_i, v_i \right\rangle _{L^2(\lambda_{(0,1)^2}; \R)}.
	\end{split}
	\end{equation*}
Therefore, we obtain that
	\begin{equation*}
	\begin{split}
	& \left\langle   v, F( v + w ) \right\rangle _H  \\
	& =  \vt \left\langle   v, v+w \right\rangle _H   + c_1 \smallsum_{j,i=1}^2 \left\langle   \partial_j v_i ,  (v_j + w_j ) w_i \right\rangle _{L^2(\lambda_{(0,1)^2}; \R)}.
	\end{split}
	\end{equation*}
	The Cauchy-Schwarz inequality yields
	\begin{equation*}
	\begin{split}
	& \left|\left\langle   v, F( v + w ) \right\rangle _H  \right|
	\\
	& \leq |\vt| \| v \|_H (\| v \|_H + \|w\|_H )   
	\\
	& \qquad + |c_1| \smallsum_{j,i=1}^2  \|\partial_j v_i \|_{L^2(\lambda_{(0,1)^2}; \R)}  \left(\|v_j w_i\|_{L^2(\lambda_{(0,1)^2}; \R)} + \|w_j w_i \|_{L^2(\lambda_{(0,1)^2}; \R)}\right)\\
	& \leq \frac32 |\vt| \| v \|_H^2+ \frac12 |\vt| \|w\|^2_H  
	 + |c_1| \left(\smallsum_{j,i=1}^2  \|\partial_j v_i \|_{L^2(\lambda_{(0,1)^2}; \R)}^2 \right)^{\nicefrac12}  
	 \\
	 & \qquad \cdot
	\left[\left(\smallsum_{j,i=1}^2 \|v_j w_i\|_{L^2(\lambda_{(0,1)^2}; \R)}^2 \right)^{\!\nicefrac12} 
	+ \left(\smallsum_{j,i=1}^2 \|w_j w_i \|_{L^2(\lambda_{(0,1)^2}; \R)}^2\right)^{\!\nicefrac12} \right].
	\end{split}
	\end{equation*}
	Furthermore the fact that for all $a,b\in \mathbb{R}$,  $ 2 a b \leq \varepsilon a^2 +\tfrac{b^2}{\varepsilon}$, together with Item~\eqref{eq:NS:ibp:gradient_norm} in Lemma~\ref{lem:NS:ibp}, yields
	\begin{equation} \label{eq:coer:NS:preliminary}
	\begin{split}
	& \left| \left\langle   v, F( v + w ) \right\rangle _H \right| 
	\\
	& \leq \frac32 |\vt| \| v \|_H^2+ \frac12 |\vt| \|w\|^2_H  + \varepsilon \smallsum_{j,i=1}^2  \|\partial_j v_i \|_{L^2(\lambda_{(0,1)^2}; \R)}^2 \\
	& \quad + \frac{c_1^2}{2\varepsilon} \smallsum_{j,i=1}^2 \|v_j w_i\|_{L^2(\lambda_{(0,1)^2}; \R)}^2  + \frac{c_1^2}{2\varepsilon} \smallsum_{j,i=1}^2 \|w_j w_i \|_{L^2(\lambda_{(0,1)^2}; \R)}^2\\
	& \leq \frac32 |\vt| \| v \|_H^2+ \frac12 |\vt| \|w\|^2_H  + 2 \varepsilon \|v\|_{H_{\nicefrac12}}^2 \\
	& \quad + \frac{c_1^2}{2\varepsilon} \smallsum_{j,i=1}^2 \|v_j w_i\|_{L^2(\lambda_{(0,1)^2}; \R)}^2  + \frac{c_1^2}{2\varepsilon} \smallsum_{j,i=1}^2 \|w_j w_i \|_{L^2(\lambda_{(0,1)^2}; \R)}^2.
	\end{split}
	\end{equation}
	Furthermore observe that 
	\begin{equation*}
	\begin{split}
	& \sum_{j,i=1}^2 \| v_j w_i\|_{L^2(\lambda_{(0,1)^2}; \R)}^2 
	= \sum_{j,i=1}^2 \int_{(0,1)^2} |v_j(x)|^2 |w_i(x)|^2 d x\\
	& \leq \left[\sup\nolimits_{x \in (0,1)^2} |\underline{w}(x)|_2^2\right] \sum_{j=1}^2 \int_{(0,1)^2} |v_j(x)|^2 d x
	=  \left[\sup\nolimits_{x \in (0,1)^2} |\underline{w}(x)|_2^2\right] \|v\|_H^2,
	\end{split}
	\end{equation*}
	$\|w\|^2_H \leq \left[\sup\nolimits_{x \in (0,1)^2} |\underline{w}(x)|_2^2\right] $ (that is well defined due to the fact that $w\in H_{\rho}$ and, e.g., Lemma~\ref{lem:NS:cont}),
	and 
	$\sum_{j,i=1}^2 \| w_j w_i\|_{L^2(\lambda_{(0,1)^2}; \R)}^2  \leq \left[\sup\nolimits_{x \in (0,1)^2} |\underline{w}(x)|_2^4\right]
	$.
	This, together with \eqref{eq:coer:NS:preliminary}, completes the proof of Lemma \ref{coer:NS}.
\end{proof}


\begin{lemma}[Lipschitzianity on bounded sets] \label{lem:locLip}
	Assume Setting \ref{sett:F} 
	and let $\theta\in [0,\infty]$ satisfy 
	\begin{equation*}
	\theta = \max\left\{|\vt| \left[\sup\nolimits_{ u\in {H_{\rho}}\setminus\{0\} } \tfrac{\|u\|_{H}} {\|u\|_{H_{\rho}}}\right] , 4 |c_1| \left[ \smallsum_{h\in \mathbb{H}} (\lambda_h)^{-2\rho} \right]^{\! \nicefrac12} \right\}.
	\end{equation*}
	Then $\theta \in [0,\infty)$, $F \in C(H_{\rho},H)$, and for all $v,w\in H_{\rho}$ it holds that
	\begin{equation*}
	\|F(v)-F(w)\|_H \leq \theta (1+\|v\|_{H_{\rho}} + \|w\|_{H_{\rho}})\|v-w\|_{H_{\rho}}
	<\infty.
	\end{equation*}
\end{lemma}
\begin{proof}[Proof of Lemma \ref{lem:locLip}.]
	Throughout this proof let $v=(v_1,v_2), w=(w_1,w_2) \in H_{\rho}$ be fixed.
	First, note that 
	\begin{equation*}
	\begin{split}
	F(v)-F(w) &=\vt (v-w) - c_1 \smallsum_{j=1}^2 \left(v_j \partial_j v - w_j \partial_j w\right) \\
	& = \vt (v-w) - c_1 \smallsum_{j=1}^2 \left((v_j-w_j) \partial_j v + w_j \partial_j (v-w)\right).
	\end{split}
	\end{equation*}
	Triangle inequality, the fact that $H_{\rho} \subseteq H$, and that $R$ is an orthogonal projection 
	yield
	\begin{equation} \label{eq:LocalLip:1}
	\begin{split}
	& \|F(v)-F(w)\|_{H} 
	\\
& \leq 
	|\vt| \|v-w\|_{H} + |c_1| \left\|R\smallsum_{j=1}^2 (v_j-w_j) \partial_j v \right\|_{H}+ |c_1| \left\| R\smallsum_{j=1}^2 w_j \partial_j (v-w) \right\|_{H}
	\\
	& \leq 
	|\vt| \left[\sup\nolimits_{ u\in {H_{\rho}}\setminus\{0\} } \frac{\|u\|_{H}} {\|u\|_{H_{\rho}}}\right] \|v-w\|_{H_{\rho}} 
	\\
	& \quad + |c_1| \left\|\smallsum_{j=1}^2 (v_j-w_j) \partial_j v \right\|_{U}+ |c_1| \left\| \smallsum_{j=1}^2 w_j \partial_j (v-w) \right\|_{U} \!\!.
	\end{split}
	\end{equation}
	Furthermore note that triangle inequality and the fact that for all $x=(x_1,x_2) \in \R^2$ it holds that $\max\{|x_1|,|x_2|\}\leq |x|_2 $ establish for all $u,u' \in H_{\rho}$ that
	\begin{equation} \label{eq:LocalLip:2}
	\begin{split}
	\left\| \smallsum_{j=1}^2 u_j \partial_j (u') \right\|_{U} & 
	\leq 
	\smallsum_{j=1}^2 \left\| u_j \partial_j (u') \right\|_{U}
	\\
	& \leq
	\smallsum_{j=1}^2 \left\| u_j \right\|_{L^{\infty}(\lambda_{(0,1)^2}; \R)} \left\| \partial_j (u') \right\|_{U}
	\\
	& \leq
	\left\| u \right\|_{L^{\infty}(\lambda_{(0,1)^2}; \R^2)} \left[\smallsum_{j=1}^2\left\| \partial_j (u') \right\|_{U} \right].
	\end{split}
	\end{equation}
Combining \cite[Item~(ii) in Lemma~4.4]{jentzen2016exponential} with Items~\eqref{item:eigenf:5} and \eqref{item:eigenf:6} in Lemma~\ref{lem:eigenf} shows for all $u\in H_{\rho}$, $j\in \{1,2\}$ that
	\begin{equation} \label{eq:LocalLip:3}
	\begin{split}
	\left\| \partial_j u \right\|_{U} 
	& \leq 
	\left[\sup_{h\in \mathbb{H}} {\|\partial_j h\|_U}{|\lambda_h|^{-\rho}} \right] \|u\|_{H_{\rho}} 
	\\
	& \leq 
	 \|u\|_{H_{\rho}} <\infty.
	\end{split}
	\end{equation}
	In addition
	Lemma~4.3 in \cite{jentzen2016exponential} together with Item~\eqref{item:eigenf:L4} in Lemma~\ref{lem:eigenf} and Item~\eqref{item:useful:3} in Lemma~\ref{lem:useful} (with $\varepsilon = 2 \rho -1$)  
	ensure for all $u\in H_{\rho}$ that
	\begin{equation} \label{eq:LocalLip:4}
	\begin{split}
	\left\| u \right\|_{L^{\infty}(\lambda_{(0,1)^2}; \R^2)}
	& \leq 
	\left[\sup_{h\in \mathbb{H}} \|h\|_{L^{\infty}(\lambda_{(0,1)^2}; \R^2)} \right] \left[ \smallsum_{h\in \mathbb{H}} (\lambda_h)^{-2\rho} \right]^{\! \nicefrac12} \|u\|_{H_{\rho}} 
	\\ & \leq 
	2 \left[ \smallsum_{h\in \mathbb{H}} (\lambda_h)^{-2\rho} \right]^{\! \nicefrac12} \|u\|_{H_{\rho}} <\infty .
	\end{split}
	\end{equation}
	Inequalities \eqref{eq:LocalLip:2}-\eqref{eq:LocalLip:4} show for all $u,u' \in H_{\rho}$ that
	\begin{equation} \label{eq:LocalLip:5}
	\begin{split}
	\left\| \smallsum_{j=1}^2 u_j \partial_j (u') \right\|_{U} 
	& \leq  
	4 \left[ \smallsum_{h\in \mathbb{H}} (\lambda_h)^{-2\rho} \right]^{\! \nicefrac12} \|u\|_{H_{\rho}} \|u'\|_{H_{\rho}} <\infty.
	\end{split}
	\end{equation}
	The fact that $\vt <\infty$, $H_{\rho}\subseteq H$, and Item~\eqref{item:useful:3} in Lemma~\ref{lem:useful} (with $\varepsilon = 2 \rho -1$) 
	imply that	$\theta \in [0,\infty)$.
	Finally \eqref{eq:LocalLip:5} and \eqref{eq:LocalLip:1} yield
	\begin{equation*}
	\begin{split}
	\|F(v)-F(w)\|_{H} 
	& \leq
	|\vt| \left[\sup\nolimits_{ u\in {H_{\rho}}\setminus\{0\} } \frac{\|u\|_{H}} {\|u\|_{H_{\rho}}}\right] \|v-w\|_{H_{\rho}} 
	\\
	& 
	\quad + 4 |c_1| \left[ \smallsum_{h\in \mathbb{H}} (\lambda_h)^{-2\rho} \right]^{\! \nicefrac12} (\left\| v \right\|_{H_{\rho}} +\left\| w \right\|_{H_{\rho}} )
	\left\| v-w \right\|_{H_{\rho}} 
	\\
	& \leq \theta  (1+\left\| v \right\|_{H_{\rho}} +\left\| w \right\|_{H_{\rho}} )
	\left\| v-w \right\|_{H_{\rho}}  <\infty .
	\end{split}
	\end{equation*}
	The proof of Lemma~\ref{lem:locLip} is thus completed.
\end{proof}

\section{Properties of the stochastic convolution process} \label{sec:noise}

This section is dedicated to check the assumptions on the stochastic convolution process.

\begin{sett} \label{sett:noise}
Assume Setting~\ref{sett:operator},
let
$ T \in (0,\infty)$,
$\rho\in (\nicefrac12,1)$,
$\varrho\in (\rho, 1)$,
$\gamma, \delta \in (\varrho,\infty)$,  
let $\left( h_n \right)_{n\in\N} \subseteq (0, T]$ satisfy that $ \limsup_{ m \to \infty} h_m =0$,
let $\xi \in H_{\gamma} $,
let $ ( \Omega, \F, \P ) $ be a probability space  with a normal filtration $(\mathbb{F}_{t})_{t \in [0,T]}$,
and let $(W_t)_{t \in [0, T]}$ be an $\Id_H$-cylindrical $( \Omega, \F, \P, (\mathbb{F}_{t})_{t \in [0,T]} )$-Wiener process
\end{sett}

\begin{remark}[Trace class additive noise]
The additive noise we are considering is actually a $(-A)^{-2\delta}$-Wiener process on the separable Hilbert space H (c.f. Section 4.1.1. in \cite{da2014stochastic}).
However, in what follows, we prefer to keep expressing the noise in terms of a $\Id_H$-cylindrical Wiener process and the constant diffusion coefficient $(-A)^{-\delta}$.
\end{remark}

\begin{lemma}[Strong convergence rates]\label{lemma:conv:rate}
	Assume Setting~\ref{sett:noise}, let  $p \in [2, \infty)$, $n \in \N$, $\varepsilon \in [0, \delta -\varrho)$,  
	let $O \colon  [0, T] \times \Omega \to H_{\varrho}$ and 
	$\mathcal{O}^n \colon  [0, T] \times \Omega \to P_n(H)$ be stochastic process processes, and assume for all $t \in [0, T]$ that 
	$ [O_t ]_{\P, \mathcal{B}(H)} =  \int_0^t e^{(t-s)A} \, (-A)^{-\delta} \, dW_s$ and $ [\mathcal{O}^n_t ]_{\P, \mathcal{B}(H)} =  \int_0^t P_n \, e^{(t-s)A} \, \, (-A)^{-\delta} \, dW_s$. Then 
	\begin{equation*}
	\sup_{t \in [0, T]} \left(\E\!\left[ \|O_t  - \mathcal{O}_t^n \|_{H_{\varrho}}^p \right] \right)^{\nicefrac1p} \leq  \tfrac{ \sqrt{p(p-1)} }{ 2 \, | 4 \pi^2|^{-\varepsilon} } \left[ \sum_{h \in \mathbb{H}} \tfrac{ (\kappa+\lambda_h)^{2\varrho+2\varepsilon}}{(\lambda_h)^{1+2\delta} } \right]^{\nicefrac{1}{2}} n^{-2\varepsilon}<\infty.
	\end{equation*}
\end{lemma}
\begin{proof}[Proof of Lemma~\ref{lemma:conv:rate}.]
	First, observe that the Burkholder-Davis-Gundy type inequality in Theorem~4.37 in~\cite{da2014stochastic}  implies  for all  $t \in [0, T]$   that
	\begin{equation}\label{eq:burk}
	\begin{split}
	&\left(\E\!\left[ \|O_t - \mathcal{O}_t^n\|_{H_{\varrho}}^p\right]\right)^{\nicefrac1p}  \\
	& = \left\| \smallint_0^t (\mathrm{Id}_H- P_n) \, e^{(t-s)A} (-A)^{-\delta} \,  d W_s \right\|_{L^p(\P; H_{\varrho})} \\
	& \leq \left[ \tfrac{p(p-1)}{2} \smallint_0^t \big \|(\mathrm{Id}_H- P_n) \,  e^{(t-s)A} (-A)^{-\delta} \big\|^2_{\operatorname{HS}(H, H_{\varrho})} \, ds \right]^{\nicefrac{1}{2}}.
	\end{split}
	\end{equation}
	Next note that Item~\eqref{item:useful:1} in Lemma~\ref{lem:useful} and Fatou's Lemma 
	imply for all  $ t \in [0, T] $ that
	\begin{equation*}
	\begin{split}
	& \int_0^t \big \|(\mathrm{Id}_H- P_n) \, e^{(t-s)A} (-A)^{-\delta}\big\|^2_{\operatorname{HS}(H, H_{\varrho})} \, ds \\
	& \leq  
	\int_0^t  \|\mathrm{Id}_H- P_n \|^2_{L(H_{\varrho+\varepsilon}, H_{\varrho})} \, \| e^{(t-s)A} (-A)^{-\delta} \|^2_{\operatorname{HS}(H, H_{\varrho+\varepsilon})} \, ds\\ 
	& = 
	\| ( \kappa - A )^{ - \varepsilon } ( \Id_H - P_n ) \|^2_{ L( H ) }
	 \left[  \int_0^t \sum_{h\in \mathbb{H}} (\kappa+\lambda_h)^{2\varrho+2\varepsilon} \lambda_h^{-2\delta} e^{-2\lambda_h s} \, ds \right] \\ 
	& \leq 
	|\kappa+\kt + 4 \pi^2 n^2 |^{-2\varepsilon} \left[ \liminf_{m\to \infty} \sum_{h \in \mathbb{H}_m} \tfrac{ (\kappa+\lambda_h)^{2\varrho+2\varepsilon}}{2 (\lambda_h)^{1+2\delta} } \right]\leq | 4 \pi^2 n^2 |^{-2\varepsilon}  \left[ \sum_{h \in \mathbb{H}} \tfrac{ (\kappa+\lambda_h)^{2\varrho+2\varepsilon}}{2 (\lambda_h)^{1+2\delta} } \right].
	\end{split}
	\end{equation*}
	This together with \eqref{eq:burk} yields  for all  $t \in [0, T]$  that
	\begin{equation*}
	\left(\E\!\left[ \|O_t - \mathcal{O}_t^n\|_{H_{\varrho}}^p\right]\right)^{\!\nicefrac1p} 
	\leq \left(\tfrac{p(p-1) | 4 \pi^2|^{-2\varepsilon}}{4}\right)^{\!\nicefrac12} \left[\sum_{h \in \mathbb{H}} \tfrac{ (\kappa+\lambda_h)^{2\varrho+2\varepsilon}}{(\lambda_h)^{1+2\delta} } \right]^{\nicefrac{1}{2}} n^{-2\varepsilon}.
	\end{equation*}
	Finally, the fact that that $\delta>\varrho+\varepsilon$ and Item~\eqref{item:useful:2} in Lemma~\ref{lem:useful} ensure that $\sum_{h\in \mathbb{H}} \lambda_h^{-(1+2\delta)} (\kappa+\lambda_h)^{2(\varrho+\varepsilon)} <\infty$. 
	The proof of Lemma~\ref{lemma:conv:rate} is thus completed.
\end{proof}


\begin{lemma}\label{lem:coerc_finite}
	Assume Setting~\ref{sett:noise}, 
	let $\beta \in (0,\nicefrac12)$, $p \in (\nicefrac{2}{\beta}, \infty)$, $t \in [0,T]$, $n \in \N$, $\eta \in [0,\infty)$, 
	let $Y \colon \Omega \to \R$ be a standard normal random variable, 
	let $ \mathbb{O}_t \colon \Omega \to P_n(H)$ be an $\F \slash \mathcal{B}(P_n(H))$-measurable function, 
	and assume  that  $\left[\mathbb{O}_t \right]_{\P, \mathcal{B}(H)} = \int_0^t P_n \, e^{(t-s)(A-\eta)} (-A)^{-\delta} \, dW_s$. 
	Then 
	\begin{equation*}
	\begin{split}
	& \left(\E \! \left[\sup\nolimits_{x \in (0,1)^2} | \und{\mathbb{O}_t}(x) |_2^2 \right] \right)^{\!\nicefrac{1}{2}} \\
	& \leq  
	\Big[\sup \! \Big(\Big\{ \! \sup\nolimits_{x \in (0,1)^2} |v(x)|_2 \colon \big[v \in \mathcal{C}((0,1)^2, \R^2) \text{ and } \|v\|_{\W^{\beta, p}((0,1)^2, \mathbb{R}^2)} \leq 1\big] \Big\}\Big)\Big] \\
	& \quad \cdot   16 \big( \E \big[|Y|^p \big] \big)^{\nicefrac{1}{p}} \left[  \smallsum_{h \in \mathbb{H}_n} \tfrac{ \max\{1,\lambda_{h}^{2 \beta}\} \lambda_{h}^{-2\delta} }{ \lambda_{h}+\eta}  \right]^{\nicefrac12}< \infty.
	\end{split}
	\end{equation*}
\end{lemma}
\begin{proof}[Proof of Lemma~\ref{lem:coerc_finite}.] 
Throughout this proof let $I,J \subseteq \Z^2 \times \{0,1\}$ be the sets which satisfy $J=\{(0,0,1)\} \cup \{(k,l,0) \colon k,l\in \Z \text{ and } |(k,l)|_2 \leq n\}$ and $I=J\setminus\{(0,0,0),(0,0,1)\}=\{(k,l,0) \colon (k,l)\in \Z^2\setminus \{(0,0)\} \text{ and }|(k,l)|_2 \leq n\}$. Then $\mathbb{H}_n =\{e_k \colon k\in J\}$.
	First, note that it holds that	
	\begin{equation}\label{lem:coerc_eq1}
	\begin{split}
	&\E \! \left[\sup\nolimits_{x \in (0,1)^2} | \und{\mathbb{O}_t}(x) |_2^2 \right]^{\nicefrac12}\\
	&\leq \Big[\sup \! \Big(\Big\{ \! \sup\nolimits_{x \in (0,1)^2} |v(x)|_2 \colon \big[v \in \mathcal{C}((0,1)^2, \R^2) \text{ and } \|v\|_{\W^{\beta, p}((0,1)^2, \mathbb{R}^2)} \leq 1\big] \Big\}\Big)\Big] \\
	& \qquad \cdot \left(\E \! \left[\big\| \und{\mathbb{O}_t} \big\|_{\W^{\beta, p}((0,1)^2, \mathbb{R}^2)}^2 \right]\right)^{\!\!\nicefrac12} \! . 
	\end{split}
	\end{equation}
	Moreover, observe that 
	H{\"o}lder's inequality shows that
	\begin{equation*}
	\left( \E \! \left[\big\| \und{\mathbb{O}_t} \big\|_{\W^{\beta, p}((0,1)^2, \mathbb{R}^2)}^2 \right] \right)^{\!\! \nicefrac12} 
	\leq 
	\left( \E \! \left[\big\| \und{\mathbb{O}_t} \big\|_{\W^{\beta, p}((0,1)^2, \mathbb{R}^2)}^p \right]\right)^{\!\!\nicefrac{1}{p}}
	\end{equation*}
	and the fact that $p\geq 2$ and the fact that for all $v=(v_1,v_2) \in \R^2$ it holds that 
	$
		|v|^p_2= \left(\sum_{j=1}^2 v_j^2 \right)^{\!\nicefrac{p}{2}} \leq 2^{\frac{p}{2}-1} \sum_{j=1}^2  |v_j|^p
	$ show that
	\begin{equation*}
	\begin{split}
	&
	\E \! \left[\big\| \und{\mathbb{O}_t} \big\|_{\W^{\beta, p}((0,1)^2, \mathbb{R}^2)}^p \right] \\
	& = 
	\E \! \left[ \iint_{(0,1)^2} |\und{\mathbb{O}_t} (x)|_2^p \, dx + \iint_{(0,1)^2} \iint_{(0,1)^2} \tfrac{|\und{\mathbb{O}_t} (x)- \und{\mathbb{O}_t}(y)|_2^p}{|x-y|_2^{1+ \beta p}} \, dx \, dy\right]\\
	& \leq 
	\E \! \left[ \iint_{(0,1)^2} 2^{\frac{p}{2}-1} \smallsum_{j=1}^2  \left|\left(\und{\mathbb{O}_t} (x)\right)_j\right|^p \, dx \right]
	\\
	& \quad 
	+ \E \! \left[ \iint_{(0,1)^2} \iint_{(0,1)^2} \tfrac{2^{\frac{p}{2}-1} \sum_{j=1}^2  |(\und{\mathbb{O}_t} (x)- \und{\mathbb{O}_t}(y))_j|^p }{|x-y|_2^{1+ \beta p}} \, dx \, dy\right]\! .
	\end{split}
	\end{equation*}
	This, the fact that for every $X\colon \Omega \to \mathbb{R}$ centered normal random variable it holds that $\E \! \left[|X|^p\right]=(\E \! \left[|X|^2\right])^{\nicefrac{p}{2}} \E \! \left[|Y|^p\right]$, and the fact that for all $v=(v_1,v_2) \in \R^2$, $j\in \{1,2\}$ it holds that $v_j^2 \leq |v|_2^2$
	ensure that
	\begin{equation}\label{lem:coerc_all}
	\begin{split}
	&
	\left( \E \! \left[\big\| \und{\mathbb{O}_t} \big\|_{\W^{\beta, p}((0,1)^2, \mathbb{R}^2)}^2 \right]\right)^{\!\nicefrac{1}{2}} 
	\leq \left( \E \! \left[\big\| \und{\mathbb{O}_t} \big\|_{\W^{\beta, p}((0,1)^2, \mathbb{R}^2)}^p \right] \right)^{\!\nicefrac{1}{p}} \\
	& \leq 
	2^{\frac{1}{2}-\frac1p} \left( \E \big[|Y|^p \big] \right)^{\nicefrac1p}  \Big( \smallsum_{j=1}^2 \iint_{(0,1)^2} \E \! \left[ |(\und{\mathbb{O}_t} (x))_j|^2\right]^{\nicefrac{p}{2}} \, dx  \\
	& \qquad 
	+ \smallsum_{j=1}^2 \iint_{(0,1)^2} \iint_{(0,1)^2} \tfrac{ \E \! \left[ |(\und{\mathbb{O}_t} (x)- \und{\mathbb{O}_t}(y))_j|^2\right]^{\nicefrac{p}{2}}  }{|x-y|_2^{1+ \beta p}} \, dx \, dy \Big)^{\!\nicefrac1p}
	\\
	&  \leq 
	2^{\frac{1}{2}} \left(\E \big[|Y|^p \big] \right)^{\!\nicefrac1p}  
	\\
	& \quad \cdot \!
	\left[{\textstyle \iint}_{(0,1)^2} \left( \E \! \left[ |\und{\mathbb{O}_t} (x)|_2^2 \right] \right)^{\nicefrac{p}{2}} dx  
	+{\textstyle \iint}_{(0,1)^2}{\textstyle \iint}_{(0,1)^2} \tfrac{\left( \E \! \left[ |\und{\mathbb{O}_t} (x) - \und{\mathbb{O}_t}(y) |_2^2 \right] \right)^{\nicefrac{p}{2}}}{|x-y|_2^{1+ \beta p}}\,  dx \, dy \right]^{\nicefrac1p} \! .
	\end{split}
	\end{equation}
	Next note that the fact that $W$ is a $\Id_H$-cylindrical Wiener process ensures that for all $k,h\in \mathbb{H}$ it holds that $\left\langle e_k , W \right\rangle_U$ and $\left\langle e_h , W \right\rangle_U$ are independent.
	This, It\^o's isometry, and Item~\eqref{item:eigenf:L4} in Lemma~\ref{lem:eigenf} ensure for all $ x,y \in (0,1)^2$  that 
	\begin{equation} \label{lem:coerc_1}
	\begin{split}
	\E \! \left[ |\und{\mathbb{O}_t}(x)|_2^2 \right] 
	& = 
	\E \! \left[ \left| \smallsum\limits_{k\in J} \und{e_k} (x) \int_0^t e^{-(\lambda_{e_k}+\eta)(t-s)} \lambda_{e_k}^{-\delta}  \, d \! \left\langle    e_k, W_s \right\rangle _U  \right|_2^2 \right]\\
	& = 
	\smallsum\limits_{k\in J} \E \! \left[ \left| \und{e_k} (x) \int_0^t e^{-(\lambda_{e_k}+\eta)(t-s)} \lambda_{e_k}^{-\delta}  \, d \! \left\langle    e_k, W_s \right\rangle _U  \right|_2^2 \right]\\
	& = 
	\smallsum\limits_{k \in J}  |\und{e_k} (x)|_2^2 \int_0^t e^{-2(\lambda_{e_k}+\eta)(t-s)} \lambda_{e_k}^{-2\delta}  \, ds 
	\\
	& \leq 
	4 \smallsum\limits_{k \in J} \frac{|\und{e_k}(x)|_2^2 }{ 2(\lambda_{e_k}+\eta)} \lambda_{e_k}^{-2\delta}
	\\& 
	\leq   2 \smallsum\limits_{k \in J}  \frac{\lambda_{e_k}^{-2\delta}}{ \lambda_{e_k}+\eta}
	 =
	2 \left[ 
	\tfrac{\lambda_{e_{0,0,0}}^{-2\delta} }{\lambda_{e_{0,0,0}}+\eta} +  \tfrac{\lambda_{e_{0,0,1}}^{-2\delta} }{\lambda_{e_{0,0,1}}+\eta} + \sum_{k\in I}  \tfrac{\lambda_{e_k}^{-2\delta} }{\lambda_{e_k}+\eta} \right]
	\end{split}
	\end{equation}
	and that
	\begin{equation}\label{eq:o_diff}
	\begin{split}
	& \E \! \left[ |\und{\mathbb{O}_t} (x) - \und{\mathbb{O}_t}(y) |^2 \right] \\
	&=  
	\E \! \left[ \left| \smallsum\limits_{k\in J} \left[\und{e_k}(x) - \und{e_k}(y)\right] \int_0^t e^{-(\lambda_{e_k}+\eta)(t-s)} \lambda_{e_k}^{-\delta} \ d \! \left\langle    e_k, W_s \right\rangle _U \right|_2^2 \right] \\
	&\leq  
	\sum_{k \in J}  \tfrac{|\und{e_k}(x) - \und{e_k}(y)|_2^2 }{ 2(\lambda_{e_k}+\eta)} \lambda_{e_k}^{-2\delta}
	=
	\sum_{k\in I} \tfrac{|\und{e_k}(x) - \und{e_k}(y)|_2^2 }{ 2(\lambda_{e_k}+\eta)} \lambda_{e_k}^{-2\delta}.
	\end{split}
	\end{equation}
	Moreover, observe that for all $x=(x_1, x_2), y=(y_1, y_2) \in (0,1)^2$, $
	(k,l) \in \Z^2\setminus\{(0,0)\}$ it holds that
	\begin{equation} \label{eq:ek:diff}
	\begin{split}
	&
	|\und{e_{k,l,0}}(x) - \und{e_{k,l,0}}(y)|_2^2  \\
	& = 
	\tfrac1{k^2+l^2} \left[ l^2 \left( \varphi_k(x_1)\varphi_l(x_2) - \varphi_k(y_1)\varphi_l(y_2) \right)^2 +k^2 
	\left( \varphi_{-k}(x_1)\varphi_{-l}(x_2) - \varphi_{-k}(y_1)\varphi_{-l}(y_2) \right)^2 \right]\\
	& = 
	\tfrac{l^2}{k^2+l^2} \left( (\varphi_k(x_1)-\varphi_k(y_1))\varphi_l(x_2) + \varphi_k(y_1)(\varphi_l(x_2)-\varphi_l(y_2) ) \right)^2 \\
	& \quad + \tfrac{k^2}{k^2+l^2}
	\left( (\varphi_{-k}(x_1)-\varphi_{-k}(y_1)) \varphi_{-l}(x_2) + \varphi_{-k}(y_1)( \varphi_{-l}(x_2)-\varphi_{-l}(y_2) )\right)^2 \\
	& \leq
	\tfrac{2 l^2}{k^2+l^2} \left( (\varphi_k(x_1)-\varphi_k(y_1))^2 (\varphi_l(x_2))^2 + (\varphi_k(y_1))^2(\varphi_l(x_2)-\varphi_l(y_2) )^2 \right)\\
	& \quad + \tfrac{2 k^2}{k^2+l^2}
	\left( (\varphi_{-k}(x_1)-\varphi_{-k}(y_1))^2 (\varphi_{-l}(x_2))^2 + (\varphi_{-k}(y_1))^2 ( \varphi_{-l}(x_2)-\varphi_{-l}(y_2) )^2 \right).
	\end{split}
	\end{equation}
	In addition, note that the fact that $\beta < \nicefrac{1}{2}$, and the fact that $\forall \, x, y \in \R \colon |\sin(x)-\sin(y)|\leq |x-y|$ and $|\cos(x)-\cos(y)|\leq |x-y|$ show that for all $x,y\in (0,1)$, $k \in \Z$ it holds that
	$(\varphi_k(x))^2\leq 2$ and
	\begin{equation*}
	\begin{split}
	(\varphi_k(x)-\varphi_k(y))^2 
	& = 
	|\varphi_k(x)-\varphi_k(y)|^{2-4\beta}  |\varphi_k(x)-\varphi_k(y)|^{4\beta} \\
	& \leq (2 |\varphi_k(x)|^2 + 2 |\varphi_k(y)|^2)^{1-2\beta}  ( 2^{\nicefrac32} |k| \pi |x-y| )^{4\beta}  \\
	&\leq 
	2^{2 (1-2\beta) + 6\beta} \pi^{4\beta} |k|^{4\beta} |x-y|^{4\beta} .
	\end{split}
	\end{equation*}
	This, \eqref{eq:ek:diff}, and the fact that $2\beta<1$ implies for all $a,b\in \R$ that $a^{4\beta}+b^{4\beta} \leq 2^{1-2\beta} (a^2+b^2)^{2\beta}$
	demonstrate that for all $x=(x_1, x_2), y=(y_1, y_2) \in (0,1)^2$, $(k,l) \in \Z^2\setminus\{(0,0)\}$ it holds that
	\begin{equation*}
	\begin{split}
	|\und{e_{k,l,0}}(x) - \und{e_{k,l,0}}(y)|_2^2   
	& \leq
	4 \, 2^{2(1+\beta)} \, \tfrac{ k^2+ l^2}{k^2+l^2} \, \pi^{4\beta} \left(  |k|^{4\beta} |x_1-y_1|^{4\beta}    + |l|^{4\beta} |x_2-y_2|^{4\beta} \right)\\
	& \leq
	2^{2 (2+\beta) }  \pi^{4\beta} (k^{4\beta}+l^{4\beta}) \left(  |x_1-y_1|^{4\beta}  + |x_2-y_2|^{4\beta} \right)
	\\
	& \leq 2^{2(3-\beta)}  \pi^{4\beta} (k^2+l^2)^{2\beta} \left(  |x_1-y_1|^{2}  + |x_2-y_2|^{2} \right)^{2\beta}	 
	\\
	& =  2^{6(1-\beta)}  (4 \pi^2 |(k,l,0)|^2_3)^{2\beta} |x-y|_2^{4\beta}.
	\end{split}
	\end{equation*}
	Combining this with \eqref{eq:o_diff} proves for all $x, y \in (0, 1)^2$ that
	\begin{equation} \label{eq:o_diff_bis}
	\begin{split}
	\E \! \left[ |\und{\mathbb{O}_t} (x) - \und{\mathbb{O}_t}(y) |_2^2 \right] 
	& \leq 2^{6(1-\beta)-1}  |x-y|_2^{4 \beta}  \sum_{k \in I} \tfrac{ (4 \pi^2 |k|_3^2 )^{2\beta} \lambda_{e_k}^{- 2\delta} }{ \lambda_{e_k}+\eta}\\
	& < 2^{6(1-\beta)-1}  |x-y|_2^{4 \beta}  \sum_{k \in I} \tfrac{ \lambda_{e_k}^{2\beta} \lambda_{e_k}^{- 2\delta} }{ \lambda_{e_k}+\eta} .
	\end{split}
	\end{equation}
	Combining  \eqref{lem:coerc_1}, \eqref{eq:o_diff_bis}, and the fact that $\beta p>2$ and $\beta \in (0,\nicefrac12)$ with \eqref{lem:coerc_all} shows that 
	$\beta p-1>0$, $1 +   2^{\frac{(4-6\beta)p}{2} + \frac{\beta p -1}{2}} = 1+ 2^{\frac{(4-5\beta)p-1}{2}}\leq 2\cdot 2^{\frac{(4-5\beta)p-1}{2}} \leq 2^{\frac{5 (1-\beta)p}{2}}$, and
	\begin{equation} \label{lem:coerc_last}
	\begin{split}
	&\left( \E \! \left[\big\| \und{\mathbb{O}_t} \big\|_{\W^{\beta, p}((0,1)^2, \mathbb{R}^2)}^2 \right]\right)^{\!\nicefrac{1}{2}} \\
	& \leq 
	\sqrt{2} \big( \E \big[|Y|^p \big] \big)^{\!\nicefrac{1}{p}} \Big\{  2^{\nicefrac{p}{2}} \left[  \tfrac{\lambda_{e_{0,0,0}}^{-2\delta} }{\lambda_{e_{0,0,0}}+\eta} +  \tfrac{\lambda_{e_{0,0,1}}^{-2\delta} }{\lambda_{e_{0,0,1}}+\eta} + \smallsum_{k\in I}  \tfrac{\lambda_{e_k}^{-2\delta} }{\lambda_{e_k}+\eta}\right]^{\nicefrac{p}{2}}  
	\\
	& \qquad +  2^{(5-6\beta)\frac{p}{2}}  \left[ \smallsum_{k \in I} \tfrac{ \lambda_{e_k}^{2\beta} \lambda_{e_k}^{- 2\delta} }{ \lambda_{e_k}+\eta}\right]^{\nicefrac{p}{2}}  \iint_{(0,1)^2} \iint_{(0,1)^2} |x-y|_2^{\beta p -1 } \,  dx \, dy \Big\}^{\!\nicefrac{1}{p}} \\
	& \leq 2 \big( \E \big[|Y|^p \big] \big)^{\nicefrac{1}{p}} 
	\left[  \smallsum_{k \in J} \tfrac{ \max\{1, \lambda_{e_k}^{2 \beta}\} \lambda_{e_k}^{-2\delta} }{ \lambda_{e_k}+\eta}  \right]^{\nicefrac{1}{2}} 
	\left[  1 +   2^{\frac{(4-6\beta)p}{2} + \frac{\beta p -1}{2}}  \right]^{\nicefrac{1}{p}}\\
	& \leq 2^{\frac{7-5\beta}{2}} \big( \E \big[|Y|^p \big] \big)^{\nicefrac{1}{p}} 
	\left[  \smallsum_{k \in \mathbb{H}_n} \tfrac{ \max\{1, \lambda_{e_k}^{2 \beta}\} \lambda_{e_k}^{-2\delta} }{ \lambda_{e_k}+\eta}  \right]^{\nicefrac{1}{2}} <\infty.
	\end{split}
	\end{equation}
	The fact that the latter quantity is finite is due to the fact that $\mathbb{H}_n=\{e_k \colon k\in J\}$ is a finite set and that $\lambda_{e_k} >0$ for all $k\in J$.
	Next observe that  the Sobolev embedding theorem and the assumption that $\beta p > 2$ (see, e.g., Lemma~\ref{lem:NS:cont}) ensure that
	\begin{equation*}
	\sup \! \Big(\Big\{ \! \sup\nolimits_{x \in (0,1)^2} |v(x)|_2 \colon \big[v \in \mathcal{C}((0,1)^2, \R^2) \text{ and } \|v\|_{\W^{\beta, p}((0,1)^2, \mathbb{R}^2)} \leq 1\big] \Big\}\Big) < \infty.
	\end{equation*}
	Combining this with \eqref{lem:coerc_eq1} and \eqref{lem:coerc_last}  completes the proof of Lemma~\ref{lem:coerc_finite}.
\end{proof}


\begin{lemma}[Existence of a continuous version] \label{lem:exist:O}
Assume Setting~\ref{sett:noise} and
	let $p \in [1,\infty)$.
	Then 
	\begin{enumerate}[(i)]
			\item \label{item:exist:O:KS} 
		it holds for all $\varepsilon \in (0,\min\{\nicefrac12,\delta-\varrho\})$ that
		\begin{equation*} 
		\begin{split}
		\sup_{n\in \N} \sup & \bigg(\bigg\{ \tfrac{\left\| \sum_{i=1,2} (-1)^i \int_0^{t_i} P_n \, e^{(t_i-s)A} (-A)^{-\delta} \,  dW_s \right\|_{L^p(\P; H_{\varrho})}}{(t_2-t_1)^{\varepsilon}} \colon 
		\\
		& \qquad t_1, t_2 \in [0, T], t_1 < t_2 \bigg\} \cup \{0\}\bigg) <\infty
		\end{split}
		\end{equation*}
		and
			\item \label{item:exist:O:2}
		for all $n\in \N$ there exists stochastic processes with continuous sample paths $O \colon [0,T] \times \Omega \to H_{\varrho}$ and $\mathcal{O}^n \colon  [0,T] \times \Omega \to P_n(H)$
		satisfying for all $t \in [0,T]$ that 
		$[O_t]_{\P, \mathcal{B}(H)} =  \int_0^{t}  e^{(t-s)A} (-A)^{-\delta} \,  dW_s$ and that 
		$[{\mathcal{O}}^n_t]_{\P, \mathcal{B}(H)} =  \int_0^{t} P_n \, e^{(t-s)A}  (-A)^{-\delta}  \, dW_s$.
	\end{enumerate}
\end{lemma}
\begin{proof}[Proof of Lemma~\ref{lem:exist:O}]
	Throughout this proof let $\varepsilon \in (0,\min\{\nicefrac12,\delta-\varrho\})$, $q \in (\max\{p, \nicefrac{1}{\varepsilon} \}, \infty)$.
	Then observe that the Burkholder-Davis-Gundy type inequality in Theorem~4.37 in~\cite{da2014stochastic}  shows  for all $n \in \N$,  $t_1, t_2 \in [0, T]$ with $t_1 < t_2$  that
	\begin{equation*}
	\begin{split}
	&\left\| \int_0^{t_1} P_n \, e^{(t_1-s)A} (-A)^{-\delta} \,  dW_s - \int_0^{t_2}  P_n \, e^{(t_2-s)A} (-A)^{-\delta} \,  dW_s \right\|^2_{L^q(\P; H_{\varrho})}  \\
	& \quad +
	\left\| \int_0^{t_1} e^{(t_1-s)A} (-A)^{-\delta} \,  dW_s - \int_0^{t_2}  e^{(t_2-s)A} (-A)^{-\delta} \,  dW_s \right\|^2_{L^q(\P; H_{\varrho})}  \\
	& \leq 
	2 \tfrac{q(q-1)}{2} \int_{t_1}^{t_2} \big \| e^{(t_2-s)A} (-A)^{-\delta} \big\|^2_{\operatorname{HS}(H, H_{\varrho})} \, ds \\
	& \quad + 2 \tfrac{q(q-1)}{2}  \int_{0}^{t_1} \big \| (e^{(t_1-s)A}- e^{(t_2-s)A}) (-A)^{-\delta} \big\|^2_{\operatorname{HS}(H, H_{\varrho})} \, ds.
	\end{split}
	\end{equation*}
Moreover it holds for all $s, t,t_1, t_2 \in [0, T]$ with $t<t_1 < s< t_2$  that
\begin{equation*}
	\| e^{(t_2-s)A} (-A)^{-\delta} \big\|^2_{\operatorname{HS}(H, H_{\varrho})} 
	= \sum_{h\in \mathbb{H}}  (\kappa+\lambda_h)^{2\varrho} \, \lambda_h^{-2\delta} \, e^{-2 (t_2-s) \lambda_h}
\end{equation*}
and
\begin{equation*}
\begin{split}
	& \big \| (e^{(t_1-t)A}- e^{(t_2-t)A}) (-A)^{-\delta} \big\|^2_{\operatorname{HS}(H, H_{\varrho})}
	\\
	& \leq
	\big\| (\kappa-A)^{\varrho+\varepsilon} \, e^{(t_1-t) A}  \, (-A)^{-\delta} \big\|_{\operatorname{HS}(H)}^2 \, \big\|(\kappa-A)^{-\varepsilon} \big(\Id_H - e^{(t_2-t_1) A}\big) \big\|_{L(H)}^2
	\\
	&
\big\|(\kappa-A)^{-\varepsilon} \big(\Id_H - e^{(t_2-t_1) A}\big) \big\|_{L(H)}^2 \left[\sum_{h\in \mathbb{H}} (\kappa+\lambda_h)^{2(\varrho+\varepsilon)}  \, \lambda_h^{-2\delta} \, e^{-2 (t_1-t) \lambda_h}\right] \!.
\end{split}
\end{equation*}
Furthermore note that the fact that $\varepsilon<\nicefrac12$ and the fact that for all $t \in [0,\infty)$, $r \in [0,1] $ it holds that 
$
	\|(\kappa-A)^{-r} (-A)^r\|_{L(H)}\leq 1
$
and 
$
	\|(-A)^{-r}(\Id_H - e^{ t A })\|_{L(H)}\leq t^r
$ (cf., e.g., Lemma~11.36 in \cite{renardy2006introduction})
imply that for all $t_1, t_2 \in [0, T]$ with $t_1 < t_2$ 
it holds that
\begin{equation*}
\begin{split}
	& \big\|(\kappa-A)^{-\varepsilon} \big(\Id_H - e^{(t_2-t_1) A}\big) \big\|_{L(H)}\\
	& \leq \big\|(\kappa-A)^{-\varepsilon} (-A)^{\varepsilon} \big\|_{L(H)} \big\| (-A)^{-\varepsilon} \big(\Id_H - e^{(t_2-t_1) A}\big) \big\|_{L(H)}
	\leq (t_2-t_1)^{\varepsilon}.
\end{split}
\end{equation*}
The four inequalities and
Fatou's lemma assure for all   $n \in \N$,  $t_1, t_2 \in [0, T]$ with $t_1 < t_2$  that	
	\begin{equation*}
	\begin{split}
	& \left\| \int_0^{t_1} P_n \, e^{(t_1-s)A} (-A)^{-\delta} \,  dW_s - \int_0^{t_2}  P_n \, e^{(t_2-s)A} (-A)^{-\delta} \,  dW_s \right\|^2_{L^q(\P; H_{\varrho})}  \\
	& \quad +
	\left\| \int_0^{t_1} e^{(t_1-s)A} (-A)^{-\delta} \,  dW_s - \int_0^{t_2}  e^{(t_2-s)A} (-A)^{-\delta} \,  dW_s \right\|^2_{L^q(\P; H_{\varrho})} \\
	& \leq 
	q(q-1)    \bigg[ \liminf_{m\to \infty}\sum_{h\in \mathbb{H}_m} \int_{t_1}^{t_2}  (\kappa+\lambda_h)^{2\varrho} \, \lambda_h^{-2\delta} \, e^{-2 (t_2-s) \lambda_h} \, ds \\
	& \quad +  (t_2-t_1)^{2\varepsilon} \,  \liminf_{m\to \infty}\sum_{h\in \mathbb{H}_m} \int_0^{t_1} (\kappa+\lambda_h)^{2(\varrho+\varepsilon)}  \, \lambda_h^{-2\delta} \, e^{-2 (t_1-s) \lambda_h}  \,  ds  \bigg]
	\\
	& \leq
	q(q-1)    \bigg[ \smallsum_{h\in \mathbb{H}} \tfrac{(\kappa+\lambda_h)^{2\varrho} (1-e^{-2 \lambda_h (t_2-t_1)}) }{2 \lambda_h^{1+2\delta} } 
	+  (t_2-t_1)^{2\varepsilon} \, \smallsum_{h\in \mathbb{H}} \tfrac{(\kappa+\lambda_h)^{2(\varrho+\varepsilon)} (1-e^{-2 \lambda_h t_1}) }{2 \lambda_h^{1+2\delta} } \bigg]\!.
	\end{split}
	\end{equation*}
	Note that the fact that $\varepsilon \leq \nicefrac12$ and the fact that $\forall \, x\in [0,\infty)$, $r\in [0,1] $ it holds that $r\leq r^{2\varepsilon}$ and  $1-e^{-x} \leq x$ show that $0 \leq 1-e^{-x} \leq  \min\{x,1\}^{2 \varepsilon}$.
Hence, we obtain for all $n \in \N$,  $t_1, t_2 \in [0, T]$ with $t_1 < t_2$ that
	\begin{equation*}
	\begin{split}
	& \smallsum_{h\in \mathbb{H}} \tfrac{(\kappa+\lambda_h)^{2\varrho} (1-e^{-2 \lambda_h (t_2-t_1)}) }{2 \lambda_h^{1+2\delta} } 
	+  (t_2-t_1)^{2\varepsilon} \, \smallsum_{h\in \mathbb{H}} \tfrac{(\kappa+\lambda_h)^{2(\varrho+\varepsilon)} (1-e^{-2 \lambda_h t_1}) }{2 \lambda_h^{1+2\delta} }  \\
	& \leq
	\smallsum_{h\in \mathbb{H}} \tfrac{(\kappa+\lambda_h)^{2\varrho} \min\{1,2 \lambda_h (t_2-t_1)\}^{2\varepsilon} }{2 \lambda_h^{1+2\delta} } 
	+  (t_2-t_1)^{2\varepsilon} \, \smallsum_{h\in \mathbb{H}} \tfrac{(\kappa+\lambda_h)^{2(\varrho+\varepsilon)} }{2 \lambda_h^{1+2\delta} }\\
	& \leq 
	\left(1+\tfrac12\right) \left[\smallsum_{h\in \mathbb{H}} \tfrac{(\kappa+\lambda_h)^{2(\varrho+\varepsilon)} }{\lambda_h^{1+2\delta} } \right] (t_2-t_1)^{2\varepsilon}.
	\end{split}
	\end{equation*}
	Therefore, it holds for all  $n \in \N$,  $t_1, t_2 \in [0, T]$ with $t_1 < t_2$ that
	\begin{equation} \label{eq:exists:O:KC}
	\begin{split}
		& \left\| \int_0^{t_1} P_n \, e^{(t_1-s)A} (-A)^{-\delta} \,  dW_s - \int_0^{t_2}  P_n \, e^{(t_2-s)A} (-A)^{-\delta} \,  dW_s \right\|_{L^q(\P; H_{\varrho})}  \\
	& \quad +
	\left\| \int_0^{t_1} e^{(t_1-s)A} (-A)^{-\delta} \,  dW_s - \int_0^{t_2}  e^{(t_2-s)A} (-A)^{-\delta} \,  dW_s \right\|_{L^q(\P; H_{\varrho})}  \\
	& \leq 
	\left(3 q(q-1)\right)^{\nicefrac{1}{2}} \left[\smallsum_{h\in \mathbb{H}} \tfrac{(\kappa+\lambda_h)^{2(\varrho+\varepsilon)} }{\lambda_h^{1+2\delta} } \right]^{\nicefrac{1}{2}} (t_2-t_1)^{\varepsilon}.
	\end{split}
	\end{equation}
	The latter quantity is finite for all $n\in \N$ because $\delta>\varrho+\varepsilon$ and Item~\eqref{item:useful:2} in Lemma~\ref{lem:useful} ensure that $\sum_{h\in \mathbb{H}} \lambda_h^{-(1+2\delta)} (\kappa+\lambda_h)^{2(\varrho+\varepsilon)} <\infty$. 
	This implies that 
	\begin{equation*} 
	\begin{split}
		\sup_{n\in \N} \sup & \bigg(\bigg\{ \tfrac{\left\| \sum_{i=1,2} (-1)^i \int_0^{t_i} P_n \, e^{(t_i-s)A} (-A)^{-\delta} \,  dW_s \right\|_{L^q(\P; H_{\varrho})}}{(t_2-t_1)^{\varepsilon}} \colon 
		\\
		& \qquad t_1, t_2 \in [0, T], t_1 < t_2 \bigg\} \cup \{0\}\bigg) <\infty.
	\end{split}
	\end{equation*}
	The fact that $p<q$ establishes Item~\eqref{item:exist:O:KS}.

	Moreover note that \eqref{eq:exists:O:KC}, the Kolmogorov-Chentsov theorem,
	and the fact that $ q \varepsilon > 1$, hence demonstrate that there exist stochastic processes  $ O \colon [0, T] \times \Omega \to H_{\varrho}$, ${\mathcal{O}}^n \colon [0, T] \times \Omega \to P_n(H)$, $n \in \N$, and ${\mathbb{O}}^n \colon [0, T] \times \Omega \to P_n(H)$, $n \in \N$ with continuous sample paths which satisfy for all $n \in \N$, $t \in [0, T]$ that 
	$
	[O_t]_{\P, \mathcal{B}(H)} =  \int_0^{t}  e^{(t-s)A} (-A)^{-\delta}\,  dW_s
	$
	and
	$
	[{\mathcal{O}}^n_t]_{\P, \mathcal{B}(H)} =  \int_0^{t} P_n \, e^{(t-s)A}  (-A)^{-\delta} \, dW_s
	$.
	This establishes Item~\eqref{item:exist:O:2}. 
	The proof of Lemma~\ref{lem:exist:O} is thus completed.
\end{proof}


\begin{proposition}[Exponential integrability properties]\label{prop:prop:O}
	Assume Setting~\ref{sett:noise}, let $p\in (4,\infty)$, $\zeta \in [\nicefrac1p, \infty)$,
	let $\mathcal{O}^n \colon [0,T] \times \Omega \to P_n(H)$, $n \in \N$, and $O \colon [0,T]\times \Omega \to H_{\varrho}$ 
	be stochastic processes with continuous sample paths satisfying for all $t\in [0,T]$, $n\in \N$ that
	$[O_t]_{\P, \mathcal{B}(H)} =  \int_0^{t}  e^{(t-s)A} (-A)^{-\delta} \,  dW_s$ and that 
	$[{\mathcal{O}}^n_t]_{\P, \mathcal{B}(H)} =  \int_0^{t} P_n \, e^{(t-s)A}  (-A)^{-\delta}  \, dW_s$,
	and let $\phi, \Phi \colon H_1 \mapsto [0,\infty)$ be functions which satisfy for all $u\in H_1$ that $\phi(u)=\zeta + \zeta \left[\sup\nolimits_{x \in (0,1)^2} |\underline{u}(x)|_2^2\right]$ and $\Phi(u)= \zeta \max\left\{ 1, \left[\sup\nolimits_{x \in (0,1)^2} |\underline{u}(x)|_2^\zeta\right]\right\}$.
	There exists $\eta \in [\kappa,\infty)$ and stochastic processes with continuous sample paths $ \mathbb{O}^n \colon [0,T] \times \Omega \to P_n(H)$, $n \in \N$, which satisfy
	\begin{enumerate}[(i)]
		\item \label{item:prop:O:2}
		for all $t \in [0,T]$, $n \in \N$ that 
		$\left[\mathbb{O}_t^n \right]_{\P, \mathcal{B}(H)} = \int_0^t P_n \, e^{(t-s)(A-\eta)} \, (-A)^{-\delta} \, dW_s$ and 
		$
		\mathbb{O}^n_t + P_n \, e^{ t (A-\eta) } \, \xi  = {\mathcal{O}}_t^n + P_n e^{ t A } \xi - \int_0^t e^{(t-s)(A-\eta)} \, \eta (\mathcal{O}_s^n + P_n e^{ s A } \xi ) \, ds 
		$
		and 
		\item \label{item:prop:O:4}
		that
		\begin{equation*}
		\begin{split}
		&\sup_{ m \in \N} \E\bigg[ \int_0^T \exp \left( \smallint_s^T p\,\phi\big(\mathbb{O}_{\lf u \rf_{h_m} }^m +  P_m  e^{\lf u \rf_{h_m} (A-\eta)} \xi \big) \, du \right) \\
		& \quad  \cdot \max\Big\{ \big|\Phi(\mathbb{O}_{ \lf s \rf_{h_m} }^m + P_m  e^{\lf s \rf_{h_m} (A-\eta)} \xi)\big|^{p/2}, \\
		& \qquad  \big\|\mathbb{O}_s^m+ P_m  e^{s (A-\eta)} \xi\big\|_H^p, 1, \smallint\nolimits_{0}^T \big\| \mathcal{O}_u^m+ P_m \,  e^{u (A-\eta)} \xi \big\|_{H_{\varrho}}^{6 p} \, du \, \Big\} \, ds \bigg]< \infty
		\end{split}
		\end{equation*}
	\end{enumerate}
\end{proposition}
\begin{proof}[Proof of Proposition~\ref{prop:prop:O}]
Let $\beta \in (\nicefrac2p, \nicefrac12)$ be fixed.
Then note that the fact that $\beta p >2$ and Sobolev embedding theorem  (see, e.g., Lemma~\ref{lem:NS:cont}) ensures that
	\begin{equation*}
	\sup \! \Big(\Big\{ \! \sup\nolimits_{x \in (0,1)^2} |v(x)|_2 \colon \big[v \in \mathcal{C}((0,1)^2, \R^2) \text{ and } \|v\|_{\W^{\beta, p}((0,1)^2, \mathbb{R}^2)} \leq 1\big] \Big\}\Big)<\infty.
	\end{equation*}
	Next observe that for all $\eta \in [0,\infty)$ it holds that
	\begin{equation*}
	\begin{split}
	\sum_{h \in \mathbb{H}} \tfrac{  \max\{1,\lambda_{h}^{2 \beta}\} \lambda_{h}^{-2\delta} }{ \lambda_{h}+\eta} 
	& = \tfrac{  \max\{1,\kt^{2 \beta}\} \kt^{-2\delta} }{ \kt+\eta} + \sum\nolimits_{h \in \mathbb{H} \setminus \{e_{0,0,0}\}} \tfrac{ \lambda_{h}^{2 \beta} \lambda_{h}^{-2\delta} }{ \lambda_{h}+\eta} \\
	& \leq  (\min\{1,\kt \})^{-(1+2\delta)} + \sum\nolimits_{h \in \mathbb{H}\setminus \{e_{0,0,0}\}}  \lambda_{h}^{2 \beta-1-2\delta}.
	\end{split}
	\end{equation*}
	Item~\eqref{item:useful:3} in Lemma~\ref{lem:useful} and the fact that $\kt >0$ assure that the latter quantity is finite.
	Hence
	\begin{equation*}
	\limsup_{\eta \to \infty} \left[ \sum_{h \in \mathbb{H}} \tfrac{  \max\{1,\lambda_{h}^{2 \beta}\} \lambda_{h}^{-2\delta} }{ \lambda_{h}+\eta}  \right]
	= \sum_{h \in \mathbb{H}} \limsup_{\eta\to\infty} \left[ \tfrac{  \max\{1,\lambda_{h}^{2 \beta}\} \lambda_{h}^{-2\delta} }{ \lambda_{h}+\eta}\right]= 0
	\end{equation*}
	that means that there exists $\eta\in[\kappa,\infty)$ such that
		\begin{equation} \label{eq:prop:O:eq1}
		\begin{split}
		& 720  p^3 T  \zeta \ 2^{8} \left[  \sum_{h \in \mathbb{H}_n} \tfrac{ \max\{1,\lambda_{h}^{2 \beta}\} \lambda_{h}^{-2\delta} }{ \lambda_{h}+\eta}  \right] \\
		& \cdot \!\!\Big[\sup \! \Big(\Big\{ \! \sup_{x \in (0,1)^2} |v(x)|_2 \colon \big[v \in \mathcal{C}((0,1)^2, \R^2) \text{ and } \|v\|_{\W^{\beta, p}((0,1)^2, \mathbb{R}^2)} \leq 1\big] \Big\}\Big)\Big]^2 \leq 1.
		\end{split}
		\end{equation}
	
	From now on let $\eta\in [\kappa, \infty)$ be fixed.
	Then, for all $n\in \N$ let ${{Q}}^n \colon [0, T] \times \Omega \to P_n(H)$ be the function satisfying for all $t\in [0,T]$
	that 
	\begin{equation} \label{eq:prop:O:Q}
	{Q}^n_t = {\mathcal{O}}_t^n + P_n e^{ t A } \xi - \int_0^t e^{(t-s)(A-\eta)} \, \eta (\mathcal{O}_s^n + P_n e^{ s A } \xi ) \, ds.
	\end{equation}
	This defines stochastic processes with continuous sample paths	
	%
	%
	%
	Moreover Proposition~5.1 in \cite{jentzen2017strong} 
	(with 
	$\alpha=\beta=\gamma=0$, 
	$O=P_n(H)$, $F=(P_n(H) \ni v \mapsto 0 \in H)$, $\tilde F=(P_n(H) \ni v \mapsto \eta v \in H)$, 
	$B=(P_n(H) \ni v \mapsto (H \ni u \mapsto P_n (-A)^{-\delta} u ) \in \operatorname{HS}(H))$,
	$\xi=(\Omega \ni \omega \mapsto P_n \xi \in P_n(H))$, 
	$X=([0,T]\times\Omega \ni (t,\omega)\mapsto (\mathcal{O}^n_t(\omega)+P_n e^{ t A } \xi)\in P_n(H))$ 
	for $n \in \N$ in the notation of Proposition~5.1 in \cite{jentzen2017strong})
	ensures that for all $n \in \N, t\in [0,T]$ it holds that
	\begin{equation*}
	\begin{split}
	\left[ {\mathcal{O}}_t^n + P_n e^{ t A } \xi \right]_{\P, \mathcal{B}(H)} 
	& = 
	\left[ P_n e^{ t (A-\eta) } \xi + \int_0^t e^{(t-s)(A-\eta)} \, \eta (\mathcal{O}_s^n + P_n e^{ s A } \xi ) \, ds \right]_{\P, \mathcal{B}(H)} 
	\\
	&+  
	\int_0^{t} P_n \, e^{(t-s)(A-\eta)}  (-A)^{-\delta} \, dW_s.
	\end{split}
	\end{equation*} 
	This and \eqref{eq:prop:O:Q} demonstrate for all $n \in \N, t\in [0,T]$ it holds that 
	$[{Q}^n_t - P_n e^{ t (A-\eta) } \xi ]_{\P, \mathcal{B}(H)} =  \int_0^{t} P_n \, e^{(t-s)(A-\eta)}  (-A)^{-\delta} \, dW_s $. 
	Choosing 
	$\mathbb{O}^n \colon [0,T]\times \Omega \to P_n(H) $, $n\in \N$, be functions which satisfies for all $n\in \N,t\in [0,T]$ that
	$\mathbb{O}^n_t = {Q}^n_t - P_n \, e^{ t (A-\eta) } \, \xi $
	demonstrates Item~\eqref{item:prop:O:2}.

	
	Moreover note that for all standard normal random variables $Y \colon \Omega \to \R$ Burkholder-Davis-Gundy inequality imply that $\E\left[|Y|^p\right]^{\nicefrac{2}{p}} \leq \frac{p(p-1)}{2} \leq \frac12 p^2$.
	Markov's inequality, Lemma~\ref{lem:coerc_finite}, and \eqref{eq:prop:O:eq1} imply for all $n \in \N$, $t \in [0,T]$  that
	\begin{equation*}
	\begin{split}
	& \P \! \left( \sup\nolimits_{x \in (0,1)^2} | \und{\mathbb{O}_t^n}(x) |_2^2 \geq \tfrac{1}{72 p  T  \zeta} \right) 
	\\
	& \leq 72  p T  \zeta \,  \E \! \left[\sup\nolimits_{x \in (0,1)^2} | \und{\mathbb{O}_t^n}(x) |_2^2  \right]  \\
	& \leq  72  p^3 T \zeta  \ 2^{8} \left[  \smallsum_{h \in \mathbb{H}_n} \tfrac{ \max\{1,\lambda_{h}^{2 \beta}\} \lambda_{h}^{-2\delta} }{ \lambda_{h}+\eta}  \right]\\
	& \quad \cdot \Big[\sup \! \Big(\Big\{ \! \sup\nolimits_{x \in (0,1)^2} |v(x)|_2 \colon \big[v \in \mathcal{C}((0,1)^2, \R^2) \text{ and } \|v\|_{\W^{\beta, p}((0,1)^2, \mathbb{R}^2)} \leq 1\big] \Big\}\Big)\Big]^2 
	\\
	& \leq \tfrac{1}{10}.
	\end{split}
	\end{equation*}
	Therefore Fernique's Theorem in \cite[Proposition~4.13]{jentzen2017strong} 
	(with 
	$V=P_n(H)$, 
	$\|\cdot\|_V=\left(P_n(H)\ni v\mapsto \sup_{x \in (0,1)^2 } |\underline{v}(x)|_2 \in [0,\infty)\right)$, 
	$X=O_t^n$, 
	$R=(72 p T \zeta)^{-\nicefrac12}$ for $t\in [0,T]$, $n \in \N$)
	 shows  for all $n \in \N$, $t \in [0, T]$  that
	\begin{equation}\label{eq:prop:O:13}
	\E \! \left[ \exp \! \left(4 p T \zeta  \Big\{ \! \sup\nolimits_{x \in (0,1)^2} | \und{\mathbb{O}_t^n}(x) |_2^2\Big\} \right)\right] \leq 13.
	\end{equation}

	Let us now prove Item~\eqref{item:prop:O:4}.
	First note that Fubini theorem together with Jensen's inequality ensure for all  $n \in \N$ that
	\begin{equation*}
	\begin{split}
	&\bigg(\E\bigg[ \int_0^T \exp \left( \smallint_s^T p\,\phi\big( {Q}_{\lf u \rf_{h_n} }^n\big) \, du \right) \\
	& \quad \cdot 
	\max\Big\{ 1, \big|\Phi({Q}_{ \lf s \rf_{h_n} }^n)\big|^{p/2}, \big\|{Q}_s^n\big\|_H^p, \smallint\nolimits_{0}^T \big\| \mathcal{O}_u^n+ P_n  e^{u (A-\eta)} \xi \big\|_{H_{\varrho}}^{6p} \, du \, \Big\} \, ds \bigg]\bigg)^{\! 2} \\
	& = 
	\bigg( \int_0^T  \E\bigg[\! \exp \left( \smallint_s^T p\,\phi\big( {Q}_{\lf u \rf_{h_n} }^n\big) \, du \right) \\
	& \quad \cdot 
	\max\Big\{1, \big|\Phi({Q}_{ \lf s \rf_{h_n} }^n)\big|^{p/2}, \big\|{Q}_s^n\big\|_H^p, \smallint\nolimits_{0}^T \big\| \mathcal{O}_u^n+ P_n  e^{u (A-\eta)} \xi \big\|_{H_{\varrho}}^{6p} \, du \, \Big\} \bigg] \, ds \bigg)^{\! 2} \\
	& \leq T  
	\int_0^T  \bigg(\E\bigg[\! \exp \left( \smallint_s^T p\,\phi\big( {Q}_{\lf u \rf_{h_n} }^n\big) \, du \right) \\
	& \quad \cdot 
	\max\Big\{1 \big|\Phi({Q}_{ \lf s \rf_{h_n} }^n)\big|^{p/2}, \big\|{Q}_s^n\big\|_H^p, \smallint\nolimits_{0}^T \big\| \mathcal{O}_u^n+ P_n  e^{u (A-\eta)} \xi \big\|_{H_{\varrho}}^{6p} \, du \, \Big\} \bigg]\bigg)^{\! 2} \, ds. 
	\end{split}
	\end{equation*}
Holder's inequality yield that
	\begin{equation}\label{eq:abstract:holder}
	\begin{split}
	& \bigg(\E\bigg[ \int_0^T \exp \left( \smallint_s^T p\,\phi\big( {Q}_{\lf u \rf_{h_n} }^n\big) \, du \right) \\
	& \quad \cdot 
	\max\Big\{ 1, \big|\Phi({Q}_{ \lf s \rf_{h_n} }^n)\big|^{p/2}, \big\|{Q}_s^n\big\|_H^p, \smallint\nolimits_{0}^T \big\| \mathcal{O}_u^n+ P_n  e^{u (A-\eta)} \xi \big\|_{H_{\varrho}}^{6p} \, du \, \Big\} \, ds \bigg]\bigg)^{\! 2} \\
	& \leq
	T  \int_0^T  \E\bigg[\! \exp \left( \smallint_s^T 2p \,\phi\big( {Q}_{\lf u \rf_{h_n} }^n\big) \, du \right) \bigg] \\
	&  \quad \cdot 
	\E \bigg[ \! \max\Big\{1, \big|\Phi({Q}_{ \lf s \rf_{h_n} }^n)\big|^{p},  \big\|{Q}_s^n\big\|_H^{2p}, T \smallint\nolimits_{0}^T \big\| \mathcal{O}_u^n+ P_n  e^{u (A-\eta)} \xi \big\|_{H_{\varrho}}^{12p} \, du \, \Big\} \bigg] \, ds \\
	& \leq 
	T \, \E\bigg[ \! \exp \left( \smallint_0^T 2p \,\phi\big( {Q}_{\lf u \rf_{h_n} }^n\big) \, du \right) \bigg] \\
	&  \quad \cdot 
	\int_0^T  \E \bigg[  1 + \big|\Phi({Q}_{ \lf s \rf_{h_n} }^n)\big|^{p} + \big\|{Q}_s^n\big\|_H^{2p} + T \smallint\nolimits_{0}^T \big\| \mathcal{O}_u^n+ P_n  e^{u (A-\eta)} \xi \big\|_{H_{\varrho}}^{12p} \, du \bigg] \, ds.
	\end{split}
	\end{equation}
	Let us first show that $\sup_{n\in \N}\E\bigg[ \! \exp \left( \smallint_0^T 2p \,\phi\big( {Q}_{\lf u \rf_{h_n} }^n\big) \, du \right) \bigg]<\infty$.
	The fact that $ \forall \, x, y \in \R \colon |x+y|^2 \leq 2 x^2 + 2y^2$ yields  for all $n \in \N$ that
	\begin{equation*}
	\begin{split}
	& \E \!\left[ \exp \! \left( \smallint_0^T  2p \,\phi\big({Q}_{\lf u \rf_{h_n}}^n\big) \, du \right) \right] \\
	&= \E \!\left[ \exp \! \left( \smallint_0^T  2p\zeta +  2p\zeta  \Big\{ \!\sup\nolimits_{x \in (0,1)^2} \big|\und{\mathbb{O}_{\lf u \rf_{h_n}}^n +  P_n e^{\lf u \rf_{h_n} (A-\eta)} \xi}(x)\big|_2^2 \Big\} \, du \right) \right] \\
	& \leq \exp \!\left(   2p\zeta T + 4p\zeta \smallint_0^T  \Big\{ \!\sup\nolimits_{x \in (0,1)^2} |\und{ P_n  e^{\lf u \rf_{h_n} (A-\eta)} \xi}(x)|_2^2 \Big\} \, du \right) \\
	& \quad \cdot \E \!\left[ \exp \! \left( \smallint_0^T    4p\zeta  \Big\{ \!\sup\nolimits_{x \in (0,1)^2} \big|\und{\mathbb{O}_{\lf u \rf_{h_n}}^n} \!(x)\big|_2^2 \Big\} \, du \right) \right].
	\end{split}
	\end{equation*}
	Moreover, e.g., Lemma~2.22 in \cite{cox2013local}  and  \eqref{eq:prop:O:13} show  for all $n \in \N$ that
	\begin{equation*}
	\begin{split}
	& \E \! \left[ \exp \! \left(  \smallint_0^T 4 p \zeta \Big\{\!\sup\nolimits_{x \in (0,1)^2} \big|\und{\mathbb{O}_{\lf u \rf_{h_n}}^n} \! (x)\big|_2^2 \Big\} \,  du \right)\right] \\\
	& \leq 
	\frac{1}{T}  \int_0^T \E\! \left[ \exp \! \left(  4  p T \zeta \Big\{\!\sup\nolimits_{x \in (0,1)^2} \big|\und{\mathbb{O}_{\lf u \rf_{h_n}}^n}\!(x)\big|_2^2 \Big\} \right)\right] du \leq 13 .
	\end{split}
	\end{equation*} 
	Therefore, 
	we obtain that it holds 
	for all $n \in \N$  that
	\begin{equation}
	\label{eq:abs:phi0}
	\begin{split}
	& \E \!\left[ \exp \! \left( \smallint_0^T  2p \,\phi\big({Q}_{\lf u \rf_{h_n}}^n\big) \, du \right) \right] \\
	& \leq 
	13 \exp \!\left(   2p\zeta T + 4p\zeta \smallint_0^T  \Big\{ \!\sup\nolimits_{x \in (0,1)^2} |\und{ P_n  e^{\lf u \rf_{h_n} (A-\eta)} \xi}(x)|_2^2 \Big\} \, du \right).
	\end{split}
	\end{equation}
	Next, note that the Sobolev embedding theorem (see, e.g., Lemma~\ref{lem:NS:cont}) implies that 
	\begin{equation*}
	\sup \! \Big(\Big\{\! \sup\nolimits_{x \in (0,1)^2} |\und{v}(x)|_2 \colon \big[ v \in H_{\gamma} \text{ and } \|v\|_{H_{\gamma}} \leq 1\big] \Big\}\Big) < \infty.
	\end{equation*}
	This yields for all $s \in [0, T]$, $n \in \N$ that
	\begin{equation}
	\label{eq:abs:xi}
	\begin{split}
	&\sup\nolimits_{x \in (0,1)^2} | \und{P_n  e^{s(A-\eta)} \xi} (x) |_2 \\
	&\leq  \Big[\sup \! \Big(\Big\{ \!\sup\nolimits_{x \in (0,1)^2} |\und{v}(x)|_2 \colon \big[v \in H_{\gamma} \text{ and } \|v\|_{H_{\gamma}} \leq 1\big] \Big\}\Big) \Big]  \|P_n  e^{s(A-\eta)}  \xi \|_{H_{\gamma}}  \\
	& \leq \Big[\sup \! \Big(\Big\{ \sup\nolimits_{x \in (0,1)^2} |\und{v}(x)|_2 \colon \big[v \in H_{\gamma} \text{ and } \|v\|_{H_{\gamma}} \leq 1\big] \Big\}\Big) \Big] \| \xi \|_{H_{\gamma}} < \infty.
	\end{split}
	\end{equation}
	Combining this with \eqref{eq:abs:phi0} yields that 
	\begin{equation}\label{eq:abstract:phi}
	\sup_{n \in \N} \E \!\left[ \exp \! \left( \smallint_0^T  2p \,\phi\big({Q}_{\lf u \rf_{h_n}}^n\big) \, du \right) \right] < \infty.
	\end{equation}
	Let us show that $\sup_{n\in \N} \int_0^T  \E \big[\big|\Phi({Q}_{ \lf s \rf_{h_n} }^n)\big|^{p} + \big\|{Q}_s^n\big\|_H^{2p} \big] \, ds<\infty$.
	The triangle inequality, the fact that $p\geq 1$, $p\zeta\geq1$, and the fact that $ \forall \, x, y \in \R$, $a \in [1, \infty) \colon |x+y|^a \leq 2^{a-1} |x|^a + 2^{a-1} |y|^a$ shows for all $s \in [0, T]$, $n \in \N$ that
	\begin{equation*}
	\begin{split}
	& \E \! \left[ \big| \Phi\big( {Q}_{ \lf s \rf_{h_n } }^n\big) \big|^p \right]\\
	&=  
	\E  \bigg[ \zeta^p  \left| \max\left\{1, \Big\{ \!\sup\nolimits_{x \in (0,1)^2} \big|\und{{Q}_{ \lf s \rf_{h_n } }^n}(x)\big|_2^{\zeta} \Big\} \right\}\right|^p \bigg] \\
	&\leq  
	\E  \bigg[ 2^{p-1} \zeta^p  + 2^{p-1} \zeta^p  \Big\{ \!\sup\nolimits_{x \in (0,1)^2} \big|\und{{Q}_{ \lf s \rf_{h_n } }^n}(x)\big|_2^{p \zeta} \Big\} \bigg] \\
	& \leq  
	2^{p-1} \zeta^p +  2^{p(\zeta+1)-2} \zeta^p \, \E \! \left[  \Big\{ \!\sup\nolimits_{x \in (0,1)^2} \big|\und{\mathbb{O}_{\lf s \rf_{h_n}}^n}(x)\big|_2^{p\zeta} \Big\}
	\right.
	\\
	& \left. \quad + \Big\{ \!\sup\nolimits_{x \in (0,1)^2} \big|\und{ P_n e^{\lf s \rf_{h_n} (A-\eta)} \xi}(x)\big|_2^{p\zeta} \Big\} \right] .
	\end{split}
	\end{equation*}
	Hence, e.g., Lemma~5.7 in \cite{hutzenthaler2016strong} (with $a=4pT\zeta$, $x=\sup\nolimits_{x \in (0,1)^2} |\und{\mathbb{O}_{\lf s \rf_{h_n}}^n}(x)|_2^2$, $r=p\zeta/2$ for $s \in [0, T]$, $n \in \N$ in the notation of Lemma~5.7 in \cite{hutzenthaler2016strong}) and \eqref{eq:prop:O:13} prove for all $s \in [0, T]$, $m \in \N$ that
	\begin{equation*}
	\begin{split}
	&\E \! \left[ \big| \Phi\big( {Q}_{ \lf s \rf_{h_n } }^n\big) \big|^p \right] \leq  2^{p-1} \zeta^p+  2^{p(\zeta+1)-2} \zeta^p  \Big\{ \!\sup\nolimits_{x \in (0,1)^2} \big|\und{ P_n e^{\lf s \rf_{h_n} (A-\eta)} \xi}(x)\big|_2^{p\zeta} \Big\}   \\
	& \quad +  \tfrac{2^{p(\zeta+1)-2} \zeta^p (\lf p\zeta/2 \rf_1 +1)!}{ |4pT\zeta|^{p\zeta/2}} \, \E\! \left[  \exp \! \left( 4pT\zeta \Big\{\!\sup\nolimits_{x \in (0,1)^2} |\und{\mathbb{O}_{\lf s \rf_{h_n}}^n }\!(x)|_2^{2} \Big\} \right)  \right]\\
	&\leq 2^{p-1} \zeta^p+  2^{p(\zeta+1)-2} \zeta^p  \Big\{ \!\sup\nolimits_{x \in (0,1)^2} \big|\und{ P_n e^{\lf s \rf_{h_n} (A-\eta)} \xi}(x)\big|_2^{p\zeta} \Big\} +\tfrac{13 \cdot 2^{p(\zeta+1)-2} \zeta^p (\lf p\zeta/2 \rf_1 +1)!}{ |4pT\zeta|^{p\zeta/2}} .
	\end{split}
	\end{equation*}
	This together with \eqref{eq:abs:xi} yields that 
	\begin{equation}\label{eq:abstract:Phi}
	\begin{split}
	\sup_{n\in \N} \int_0^T \E \! \left[ \big| \Phi\big( {Q}_{ \lf s \rf_{h_n } }^n\big) \big|^p \right] ds < \infty.
	\end{split}
	\end{equation}
	Moreover, e.g., Lemma~5.7 in \cite{hutzenthaler2016strong} (with $a=4pT\zeta$, $x=\sup\nolimits_{x \in (0,1)^2} |\und{\mathbb{O}_{s}^n}(x)|_2^2$, $r=p$ for $s \in [0, T]$, $n \in \N$ in the notation of Lemma~5.7 in \cite{hutzenthaler2016strong})  and \eqref{eq:prop:O:13} ensure  for all $s \in [0, T]$, $n \in \N$  that
	\begin{equation*}
	\begin{split}
	& \E \!\left[\big\|{Q}_s^n\big\|_H^{2p} \right] \leq  \E\! \left[ 2^{2p-1}  \|\mathbb{O}_s^m\|_H^{2p} + 2^{2p-1} \| P_n  e^{s (A-\eta)} \xi\|_H^{2p} \right]\\
	& \leq 2^{2p-1} \E \! \left[\Big\{ \!\sup\nolimits_{x \in (0,1)^2} \big|\und{\mathbb{O}_{s}^n}(x)\big|_2^{2p} \Big\}\right] + 2^{2p-1} \| P_n  e^{s (A-\eta)} \xi\|_H^{2p}\\
	& \leq  \tfrac{2^{2p-1}  (\lf p \rf_1 +1)!}{ |4pT\zeta|^{p}} \E \! \left[\exp \! \left( 4pT\zeta\Big\{ \!\sup\nolimits_{x \in (0,1)^2} \big|\und{\mathbb{O}_{s}^n}(x)\big|_2^{2} \Big\} \right)\right] + 2^{2p-1} \| P_n  e^{s (A-\eta)} \xi\|_H^{2p}\\
	& \leq  \tfrac{13 \cdot 2^{2p-1}  (\lf p \rf_1 +1)!}{ |4pT\zeta|^{p}}  + 2^{2p-1} \| \xi\|_H^{2p}.
	\end{split}
	\end{equation*}
	Hence, we obtain that
	\begin{equation}\label{eq:abstract:O}
	\begin{split}
	\sup_{n \in \N} \int_0^T \E \! \left[ \big\|{Q}_s^n\big\|_H^{2p} \right] ds < \infty.
	\end{split}
	\end{equation}
	Finally, let us consider the finiteness of $\sup_{n \in \N} \E \! \left[ \smallint_{0}^T \| \mathcal{O}_u^n+ P_n  e^{u (A-\eta)} \xi \|_{H_{\varrho}}^{12p} \, du \right]$.
	Note that for all $n \in \N$ it holds that
	\begin{equation}
	\label{eq:abstract:int}
	\begin{split}
	&\E \! \left[\smallint_{0}^T \| \mathcal{O}_u^n+ P_n  e^{u (A-\eta)} \xi \|_{H_{\varrho}}^{12p} \, du  \right] \\
	&\leq 2^{12p-1} \, \E \!\left[\smallint_{0}^T \| \mathcal{O}_u^n  \|_{H_{\varrho}}^{12p}+ \|P_n \, e^{u (A-\eta)} \xi \|_{H_{\varrho}}^{12p} \, du  \right] \\
	&\leq 2^{12p-1} \, \E \!\left[\smallint_{0}^T \| \mathcal{O}_u^n  \|_{H_{\varrho}}^{12p}+ \| \xi \|_{H_{\varrho}}^{12p} \, du  \right].
	\end{split}
	\end{equation}
	Moreover, observe that the Burkholder-Davis-Gundy type inequality in Theorem~4.37 in~\cite{da2014stochastic}  implies  for all  $u \in [0, T]$, $n \in \N$  that
	\begin{equation*}
	\begin{split}
	& \E \!\left[\| \mathcal{O}_u^n  \|_{H_{\varrho}}^{12p} \right] = \E \!\left[\left\|  \int_0^u P_n \, e^{(u-s)A} \, (-A)^{-\delta}  \, dW_s \right\|_{H_{\varrho}}^{12p} \right]  \\
	&\leq \left[\tfrac{(12p)(12p-1)}{2}\right]^{6p} \left[ \int_0^u \|P_n\, e^{(u-s)A} \, (-A)^{-\delta}   \|_{\operatorname{HS}(H, H_{\varrho})}^2 \,ds\right]^{6p} \\
	& \leq \left[6p(12p-1)\right]^{6p} \left[ \int_0^u \| (\kappa-A)^{\varrho} \, P_n \,e^{(u-s)A}  \, (-A)^{-\delta}  \|_{\operatorname{HS}(H)}^2 \, ds\right]^{6p}\\
	& = \left[6p(12p-1)\right]^{6p} \left[ \sum_{h \in \mathbb{H}_n} \int_0^u (\kappa +\lambda_h)^{2\varrho} \,e^{-2\lambda_h s} \, \lambda_h^{-2\delta} \, ds\right]^{6p}\\
	& \leq \left[6p(12p-1)\right]^{6p} \left[ \sum_{h \in \mathbb{H}}  \tfrac{(\kappa +\lambda_h)^{2\varrho} (1-e^{-2\lambda_h u})}{2\lambda_h^{1+2\delta}} \right]^{6p}\\
	& \leq \left[6p(12p-1)\right]^{6p} \left[ \sum_{h \in \mathbb{H}}  \tfrac{(\kappa +\lambda_h)^{2\varrho} }{2\lambda_h^{1+2\delta}} \right]^{6p}.
	\end{split}
	\end{equation*}
	Combining this, the fact that $\delta>\varrho$, and Item~\eqref{item:useful:2} in Lemma~\ref{lem:useful} with \eqref{eq:abstract:int}  yields that
	\begin{equation*}
	\begin{split}
	\sup_{n \in \N} \E \! \left[ \smallint_{0}^T \| \mathcal{O}_u^n+ P_n  e^{u (A-\eta)} \xi \|_{H_{\varrho}}^{12p} \, du \right] < \infty.
	\end{split}
	\end{equation*}
	This, \eqref{eq:abstract:holder}, \eqref{eq:abstract:phi}, \eqref{eq:abstract:Phi}, and \eqref{eq:abstract:O} imply that
	\begin{equation*}
	\begin{split}
	&\sup_{ n \in \N} \E\bigg[ \int_0^T \exp \left( \smallint_s^T p\,\phi\big( {Q}_{\lf u \rf_{h_n} }^n\big) \, du \right) \max\Big\{ 1,\big|\Phi({Q}_{ \lf s \rf_{h_n} }^n)\big|^{p/2}, \\
	& \quad  \big\|{Q}_s^n\big\|_H^p, \smallint\nolimits_{0}^T \big\| \mathcal{O}_u^n+ P_n \, e^{u A} \xi \big\|_{H_{\varrho}}^{6p} \, du \, \Big\} \, ds \bigg]< \infty.
	\end{split}
	\end{equation*}
	This establishes Item~\eqref{item:prop:O:4}. The proof of Proposition~\ref{prop:prop:O} is thus completed.
\end{proof}


\section{Strong convergence of the approximation scheme} \label{sec:SC}

\begin{theorem} \label{th:SC}
	\label{abs:prop:last}
	Assume Setting~\ref{sett:F} and Setting~\ref{sett:noise}, 
	let $p\in (0,\infty)$ and 
	$ \chi \in  \Big(0,  \min\big\{ \tfrac{1-\rho}{5}, \tfrac{( \varrho - \rho ) }{3} \big\}\Big] $, 
	let $ X\colon [0, T] \times \Omega \to H_{\varrho}$  be a  stochastic process with continuous sample paths which satisfies for all $t \in [0, T]$  that  $ [X_t ]_{\P, \mathcal{B}(H)} =    [ e^{ t A }   \xi  + \smallint_0^t e^{  ( t - s ) A}  \, F (  X_s ) \, ds ]_{\P, \mathcal{B}(H)} + \int_0^t e^{(t-s)A} \, (-A)^{-\delta} \,  dW_s$,
	and let $\Y^n, {\mathcal{Q}}^n \colon [0,T]\times\Omega \to P_n(H)$, $n\in \N$ sequences of stochastic processes which satisfy for all $n\in \N, t\in [0,T]$ that $[{\mathcal{Q}}^n_t ]_{\P, \mathcal{B}(H)} =  \int_0^t P_n \, e^{(t-s)A} \, \, (-A)^{-\delta} \, dW_s$ and 
	\begin{equation}\label{eq:set:abstract:tildeO}
	\begin{split}
	& 1= \P \Big( 
	\Y_t^n = P_n \, e^{ t A } \, \xi  + {\mathcal{Q}}_t^n
	\\
	& \qquad
	+ \smallint_0^t P_n \,  e^{  ( t - s ) A } \, \one_{ \{ \| \Y_{ \lf s \rf_{h_n} }^n \|_{ H_{\varrho} } + \| {\mathcal{Q}}_{ \lf s \rf_{h_n} }^n +P_n  e^{ \lf s \rf_{ h_n } A } \xi \|_{ H_{\varrho} } \leq | h_n|^{ - \chi } \}} \, F \big(  \Y_{ \lf s \rf_{ h_n } }^n \big) \, ds  \Big).
	\end{split}
	\end{equation}
	Then 
	\begin{enumerate}[(i)]
		\item \label{item:SC:O} 
		there exists a sequence of stochastic processes $\mathcal{O}^n \colon [0,T]\times\Omega \to P_n(H)$, $n\in \N$, with continuous sample paths which satisfy for all $t\in [0,T]$ that $[\mathcal{O}^n_t ]_{\P, \mathcal{B}(H)} =  \int_0^t P_n \, e^{(t-s)A} \, \, (-A)^{-\delta} \, dW_s$, 
		\item \label{item:SC:X} 
		it holds for all $n \in \N, t \in [0,T]$ that
		\begin{equation*}
		\begin{split}
		& 1 =
		\P \Big( \Y_t^n = P_n \, e^{ t A } \, \xi + {\mathcal{O}}_t^n \\
		& \qquad + \smallint_0^t P_n \,  e^{  ( t - s ) A } \, \one_{ \{ \| \Y_{ \lf s \rf_{h_n} }^n \|_{ H_{\varrho} } + \| {\mathcal{O}}_{ \lf s \rf_{h_n} }^n +P_n  e^{ \lf s \rf_{ h_n } A } \xi \|_{ H_{\varrho} } \leq | h_n|^{ - \chi } \}} \, F \big(  \Y_{ \lf s \rf_{ h_n } }^n \big) \, ds \Big),
		\end{split}
		\end{equation*}
		and
		\item \label{item:SC:SC}
		it holds that
		\begin{equation*}
		\limsup_{n \to \infty} \sup_{t \in [0,T]} \E \big[ \| X_t -\Y_t^n \|_H^p \big] = 0.
		\end{equation*}
	\end{enumerate}
\end{theorem}
\begin{proof}[Proof of Theorem~\ref{abs:prop:last}.]
	Throughout this proof let 
	$\varepsilon \in (0, \min\{\delta- \varrho,\nicefrac12\})$
	and 
	$q \in (\max\{p,\nicefrac4{\varepsilon}\},\infty)$.
		Item~\eqref{item:exist:O:2} in Lemma~\ref{lem:exist:O} ensures that there exist stochastic processes $ O \colon [0, T] \times \Omega \to H_{\varrho}$ and  $ {\mathcal{O}}^n \colon [0, T] \times \Omega \to P_n(H)$, $n \in \N$, with continuous sample paths satisfying for all $n \in \N, t\in [0,T]$ that
	$ [O_t ]_{\P, \mathcal{B}(H)} =  \int_0^t e^{(t-s)A} \, (-A)^{-\delta} \, dW_s$
	and 
	$ [\mathcal{O}^n_t ]_{\P, \mathcal{B}(H)} =  \int_0^t P_n \, e^{(t-s)A} \, \, (-A)^{-\delta} \, dW_s$.
	This establishes Item~\eqref{item:SC:O}.
	
	%
	%
	%
	Next observe that the fact that for all $ n \in \N$, $t \in [0, T]$ it holds that $\P(\mathcal{O}_t^n = {\mathcal{Q}}^n_t )=1$ and \eqref{eq:set:abstract:tildeO} together with the fact that for all $n\in \N$ the processes $\mathcal{O}^n$ and $\mathcal{Q}^n$ have continuous sample paths establishes Item~\eqref{item:SC:X}.
	
	%
	%
	%
	Let us prove Item~\eqref{item:SC:SC}.
	Throughout this proof let $\zeta,\theta \in [0,\infty)$ be equal to $\zeta=\max\left\{\frac1{q},\frac32 |\vt|, \frac12 |\vt| +  2 c_1^2, 4\right\}$ and
	\begin{equation*}
	\theta = \max\left\{|\vt| \left[\sup\nolimits_{u\in H_{\rho}\setminus\{0\}} \tfrac{\|u\|_{H}} {\|u\|_{H_{\rho}}}\right], 4 |c_1| \left[\smallsum_{h\in \mathbb{H}} \lambda_h^{-2\rho} \right]^{\nicefrac{1}{2}} \right\},
	\end{equation*}
	and let $\phi,\Phi \colon H_\varrho \mapsto [0,+\infty)$ be the functions which satisfy for all $v\in H_\varrho$ that $\phi(v)=\zeta (1+ \big[\!\sup\nolimits_{x \in (0,1)^2} |\und{w}(x)|_2^2\big]) $ and $\Phi(v)=\zeta  \max\left\{1, \big[\!\sup\nolimits_{x \in (0,1)^2}  |\und{w}(x)|_2^\zeta\big] \right\}$
	(Lemma~\ref{lem:NS:cont} ensures that the latter functions are well defined). 
	First note that Lemma~\ref{lemma:conv:rate} yields
	\begin{equation*}
	\sup\nolimits_{n \in \N} \left\{n^{\varepsilon} \sup\nolimits_{t \in [0, T]} \left(\E\!\left[\|O_t  - {\mathcal{O}}_t^n \|_{H_{\varrho}}^q \right]\right)^{\nicefrac1q} \right\}< \infty.
	\end{equation*}
	This, the fact that $ O \colon [0, T] \times \Omega \to H_{\varrho}$ and  ${\mathcal{O}}^n \colon [0, T] \times \Omega \to P_n(H)$, $n \in \N$, are stochastic processes with continuous sample paths, 
	Item~\eqref{item:exist:O:KS} in Lemma~\ref{lem:exist:O}, and Corollary~2.11 in \cite{cox2016convergence} (with $T=T$, $p=q$, $\beta= \varepsilon$, $\theta^N= \{ \frac{k T}{N} \in [0, \infty) \colon k \in \N_0 \cap [0, N] \}$, $(E, \left\| \cdot \right\|_E)= (H_{\varrho}, \left\| \cdot \right\|_{H_{\varrho}})$, $Y^N= ([0,T] \times \Omega  \ni (t, \omega) \mapsto {\mathcal{O}}^N_t(\omega) \in H_{\varrho})$, $Y^0= O$, $\alpha=0$, $\varepsilon= \nicefrac{\varepsilon}{2}$ for $N \in \N$ in the notation of Corollary~2.11 in \cite{cox2016convergence}) ensure that
	\begin{equation*}
	\sup\nolimits_{n \in \N} \left( n^{(\nicefrac{\varepsilon}{2}- \nicefrac{1}{q})} \left(\E\!\left[ \sup\nolimits_{t \in [0, T]} \|O_t - {\mathcal{O}}_t^n \|_{H_{\varrho}}^q \right]\right)^{\nicefrac1q} \right)< \infty.
	\end{equation*}
	Lemma~3.21 in  \cite{hutzenthaler2015numerical} (cf., e.g., Theorem~7.12 in \cite{graham2013stochastic} and Lemma~2.1 in \cite{kloeden2007pathwise}) together with the fact that $\nicefrac{\varepsilon}{2}- \nicefrac{1}{q} > \nicefrac{1}{q}$ hence yields that  
	\begin{equation}\label{eq:O:conv}
	\P \bigg(\limsup_{n \to \infty} \sup_{s \in [0, T]} \| O_s - {\mathcal{O}}_s^n \|_{H_{\varrho}} =0 \bigg)=1.
	\end{equation}
	Next observe that the fact that $\gamma -\rho>0$ and Item~\eqref{item:useful:1} in Lemma~\ref{lem:useful} imply that it holds for all $n \in \N$, $t \in [0, T]$ that 
	\begin{equation*}
	\begin{split}
	\|(\Id_H- P_n) \, e^{ t A }  \xi \|_{H_{\varrho}} &\leq \|(\kappa-A)^{\varrho-\gamma} (\Id_H-P_n) \|_{L(H)} \|\xi\|_{H_{\gamma}} \\
	& \leq (4 \pi^2 n^2)^{-(\gamma-\varrho)} \|\xi\|_{H_{\gamma}}.
	\end{split}
	\end{equation*}
	Combining this with \eqref{eq:O:conv} proves that 
	\begin{equation*}
	\P \bigg( \limsup_{n \to \infty} \sup_{s \in [0, T]} \big\| (O_s + e^{sA}  \xi) - ({\mathcal{O}}_s^n + P_n  e^{ s A }  \xi )\big\|_{H_{\varrho}} =0 \bigg)=1.
	\end{equation*}
	Fatou's Lemma implies that
	\begin{equation}\label{eq:O:as}
	\limsup_{n \to \infty} \E \! \left[ \min\bigg\{1, \sup_{s \in [0, T]} \big\| (O_s + e^{sA}  \xi) - ({\mathcal{O}}_s^n + P_n  e^{ s A }  \xi )\big\|_{H_{\varrho}}\bigg\} \right]=0.
	\end{equation}
	Moreover note that Items~\eqref{item:prop:O:2}--\eqref{item:prop:O:4} in Proposition~\ref{prop:prop:O} show that there exists $\eta \in [\kappa,\infty)$ and a sequence of stochastic processes $\mathbb{O}^n \colon [0,T] \times \Omega \to P_n(H)$ such that it holds for all $n\in \N$, $t\in [0,T]$ that 
	\begin{equation}
	\mathbb{O}^n_t  = {\mathcal{O}}_t^n + P_n e^{ t A } \xi - \int_0^t e^{(t-s)(A-\eta)} \, \eta (\mathcal{O}_s^n + P_n e^{ s A } \xi ) \, ds 
	\end{equation} 
	and 
	\begin{equation}\label{eq:strong:limsup}
	\begin{split}
	&
	\sup_{ m \in \N} \E\bigg[ \int_0^T \exp \left( \smallint_s^T q\,\phi\big( {\mathbb{O}}_{\lf u \rf_{h_m} }^m \big) \, du \right)  \max\Big\{ 1, \big\|{\mathbb{O}}_s^m \big\|_H^q, \big|\Phi({\mathbb{O}}_{ \lf s \rf_{h_m} }^m \big|^{q/2}, 
	\\
	& \quad  \smallint\nolimits_{0}^T \big\| {\mathcal{O}}_u^m+ P_m \, e^{u (A-\eta)} \xi \big\|_{H_{\varrho}}^{6 q} \, du \, \Big\} \, ds  \bigg]  
	+ \limsup_{ m \to \infty} \sup_{ s \in [0,T]} \E[ \| {\mathbb{O}}_s^m \|_H^q]  < \infty.  
	\end{split}
	\end{equation}
	Next observe that for all $v\in H_{\nicefrac12}$ it holds that
	\[
		\|v\|_{H_{\nicefrac12}} = \|(\kappa-A)^{\nicefrac12} v\|_H 
		\leq \|(\kappa-A)^{\nicefrac12} (\eta-A)^{-\nicefrac12}\|_{L(H)} \|(\eta-A)^{\nicefrac12} v\|_H. 
	\]
The fact that $\eta \geq \kappa$ yield 
	\begin{equation*}
	\begin{split}
		\|(\kappa-A)^{\nicefrac12} (\eta-A)^{-\nicefrac12}\|_{L(H)}^2 
		& = \sup_{w\in H \colon \|w\|_H=1} \sum_{h\in \mathbb{H}} \tfrac{\kappa+\lambda_h}{\eta+\lambda_h} \langle w,h \rangle_{H}^2 
		\\
		& \leq \sup_{w\in H \colon \|w\|_H=1} \|w\|_H^2=1.
	\end{split}
	\end{equation*}
	This and Lemma~\ref{coer:NS} (with $\varepsilon =\nicefrac14$) show that it holds for all $n\in \N$, $v,w\in P_n(H)$ that $F(v+w)\in H$, and 
\begin{equation}
	\begin{split}
	\left\langle   v, P_n F( v + w ) \right\rangle _H 
	& =\left\langle   v, F( v + w ) \right\rangle _H \\
	&\leq \phi( w ) \| v \|^2_H + \tfrac12 \| v \|^2_{ H_{\nicefrac12}} + \Phi( w )\\
	&\leq \phi( w ) \| v \|^2_H + \tfrac12 \| (\eta-A)^{\nicefrac{1}{2}} v \|^2_{ H} + \Phi( w ).
	\end{split}
	\end{equation}
	Moreover Lemma~\ref{lem:locLip} ensures for all $n\in \N$, $v,w\in P_n(H)$ that 
	\begin{equation} \label{last:eq}
	\|F(v) - F(w) \|_{H}  \leq   \theta \big( 1+ \left\| v \right\|_{ H_{\rho}} +  \left\| w \right\|_{ H_{\rho}} \big) \left\| v - w \right\|_{ H_{\rho}} <\infty.
	\end{equation}
	Furthermore note that the fact that $H\subseteq H_{-1}=\overline{ H}^{H_{-1}}$ and for all $n\in \N$ it holds that $P_n\in L(H)$ implies that for all $n\in \N$ there exists an extension $R_n \in L(H_{-1},H)$ ( i.e.~such that $R_n|_H = P_n$). The fact that  $H\subseteq H_{-1}$ ensures that for all $n\in \N$ it holds that $R_n\in L(H_{-1})$.
	In addition Item~\eqref{item:useful:4} in Lemma~\ref{lem:useful} ensures that $\liminf_{ m \to \infty } \inf( \{\lambda_h \colon h \in \mathbb{H} \backslash \mathbb{H}_m \} \cup \{\infty\}  ) = \infty$.
	Hence combining  this,  \eqref{eq:O:as}--\eqref{last:eq},
	the fact that $p \in (0, q)$, the fact that $ \forall \, t \in [0, T] \colon \P(X_t = \int_0^t e^{(t-s)A} \, F(X_s) \, ds +O_t + e^{tA}  \xi)=1$, and Item~(iv) in Theorem~3.5 in \cite{jentzen2017strong} (with
	$\alpha=0$, $\varphi= \tfrac12$, 
	$p=q$, 
	$P_n=R_n$,
	  $\Y^n = ([0, T] \times \Omega \ni (\omega, t) \mapsto \Y_t^n(\omega) \in H_{\varrho})$,  
	$\mathbb{X}^n = ([0, T] \times \Omega \ni (\omega, t) \mapsto \Y_t^n(\omega) \in H_{\varrho})$,  
	$ 
	\mathcal{O}^n = ( [0,T] \times \Omega \ni (t,\omega) \mapsto ({\mathcal{O}}^n_t( \omega ) + P_n  e^{tA} \xi) \in P_n(H)) 
	$, 
	$ O = ( [0,T] \times \Omega \ni (t,\omega) \mapsto (O_t( \omega ) + e^{tA}  \xi) \in H_{\varrho}) 
	$,  $q=p$ for $n \in \N$ in the notation of Item~(iii) in Theorem~3.5 in \cite{jentzen2017strong}) establishes Item~\eqref{item:SC:SC}.
	The proof of Theorem~\ref{abs:prop:last} is thus completed. 
\end{proof}

\section*{Acknowledgements}
The author would like to thank Arnulf Jentzen and Diyora Salimova for nice discussions. 
This project has been supported by the Deutsche Forschungsgesellschaft (DFG) via RTG
2131 \emph{High-dimensional Phenomena in Probability -- Fluctuations and Discontinuity}.

\addcontentsline{toc}{section}{Bibliography}
\bibliographystyle{plainnat}
\bibliography{bibfile}

\begin{thebibliography}{23}
\providecommand{\natexlab}[1]{#1}
\providecommand{\url}[1]{\texttt{#1}}
\expandafter\ifx\csname urlstyle\endcsname\relax
  \providecommand{\doi}[1]{doi: #1}\else
  \providecommand{\doi}{doi: \begingroup \urlstyle{rm}\Url}\fi

\bibitem[Becker and Jentzen(2018)]{becker2018strong}
Sebastian Becker and Arnulf Jentzen.
\newblock Strong convergence rates for nonlinearity-truncated {E}uler-type
  approximations of stochastic {G}inzburg--{L}andau equations.
\newblock \emph{Stochastic Processes and their Applications}, 2018.

\bibitem[Becker et~al.(2017)Becker, Gess, Jentzen, and
  Kloeden]{becker2017strong}
Sebastian Becker, Benjamin Gess, Arnulf Jentzen, and Peter~E Kloeden.
\newblock Strong convergence rates for explicit space-time discrete numerical
  approximations of stochastic {A}llen-{C}ahn equations.
\newblock \emph{arXiv preprint arXiv:1711.02423}, 2017.

\bibitem[Bessaih and Millet(2018)]{bessaih2018strong}
Hakima Bessaih and Annie Millet.
\newblock On strong ${L}^{2}$ convergence of numerical schemes for the
  stochastic 2{D} {N}avier-{S}tokes equations.
\newblock \emph{arXiv preprint arXiv:1801.03548}, 2018.

\bibitem[Bessaih et~al.(2014)Bessaih, Brze{\'z}niak, and
  Millet]{bessaih2014splitting}
Hakima Bessaih, Zdzislaw Brze{\'z}niak, and Annie Millet.
\newblock Splitting up method for the 2{D} stochastic {N}avier--{S}tokes
  equations.
\newblock \emph{Stochastic Partial Differential Equations: Analysis and
  Computations}, 2\penalty0 (4):\penalty0 433--470, 2014.

\bibitem[Carelli and Prohl(2012)]{carelli2012rates}
Erich Carelli and Andreas Prohl.
\newblock Rates of convergence for discretizations of the stochastic
  incompressible {N}avier--{S}tokes equations.
\newblock \emph{SIAM Journal on Numerical Analysis}, 50\penalty0 (5):\penalty0
  2467--2496, 2012.

\bibitem[Cox et~al.(2013)Cox, Hutzenthaler, and Jentzen]{cox2013local}
Sonja Cox, Martin Hutzenthaler, and Arnulf Jentzen.
\newblock Local {L}ipschitz continuity in the initial value and strong
  completeness for nonlinear stochastic differential equations.
\newblock \emph{arXiv preprint arXiv:1309.5595}, 2013.

\bibitem[Cox et~al.(2016)Cox, Hutzenthaler, Jentzen, van Neerven, and
  Welti]{cox2016convergence}
Sonja Cox, Martin Hutzenthaler, Arnulf Jentzen, Jan van Neerven, and Timo
  Welti.
\newblock Convergence in {H}{\"o}lder norms with applications to {M}onte
  {C}arlo methods in infinite dimensions.
\newblock \emph{arXiv preprint arXiv:1605.00856}, 2016.

\bibitem[Da~Prato and Zabczyk(2014)]{da2014stochastic}
Giuseppe Da~Prato and Jerzy Zabczyk.
\newblock \emph{Stochastic equations in infinite dimensions}.
\newblock Cambridge university press, 2014.

\bibitem[D{\"o}rsek(2012)]{dorsek2012semigroup}
Philipp D{\"o}rsek.
\newblock Semigroup splitting and cubature approximations for the stochastic
  {N}avier--{S}tokes equations.
\newblock \emph{SIAM Journal on Numerical Analysis}, 50\penalty0 (2):\penalty0
  729--746, 2012.

\bibitem[Graham and Talay(2013)]{graham2013stochastic}
Carl Graham and Denis Talay.
\newblock \emph{Stochastic simulation and {M}onte {C}arlo methods: mathematical
  foundations of stochastic simulation}, volume~68.
\newblock Springer Science \& Business Media, 2013.

\bibitem[Gy{\"o}ngy et~al.(2016)Gy{\"o}ngy, Sabanis, and
  {\v{S}}i{\v{s}}ka]{gyongy2016convergence}
Istv{\'a}n Gy{\"o}ngy, Sotirios Sabanis, and David {\v{S}}i{\v{s}}ka.
\newblock Convergence of tamed {E}uler schemes for a class of stochastic
  evolution equations.
\newblock \emph{Stochastics and Partial Differential Equations: Analysis and
  Computations}, 4\penalty0 (2):\penalty0 225--245, 2016.

\bibitem[Hausenblas and Randrianasolo(2018)]{hausenblas2018time}
Erika Hausenblas and Tsiry Randrianasolo.
\newblock Time-discretization of stochastic 2-{D} {N}avier--{S}tokes equations
  with a penalty-projection method.
\newblock \emph{arXiv preprint arXiv:1805.00832}, 2018.

\bibitem[Hutzenthaler and Jentzen(2015)]{hutzenthaler2015numerical}
Martin Hutzenthaler and Arnulf Jentzen.
\newblock \emph{Numerical approximations of stochastic differential equations
  with non-globally {L}ipschitz continuous coefficients}, volume 236.
\newblock American Mathematical Society, 2015.

\bibitem[Hutzenthaler et~al.(2010)Hutzenthaler, Jentzen, and
  Kloeden]{hutzenthaler2010strong}
Martin Hutzenthaler, Arnulf Jentzen, and Peter~E Kloeden.
\newblock Strong and weak divergence in finite time of {E}uler's method for
  stochastic differential equations with non-globally {L}ipschitz continuous
  coefficients.
\newblock In \emph{Proceedings of the Royal Society of London A: Mathematical,
  Physical and Engineering Sciences}. The Royal Society, 2010.

\bibitem[Hutzenthaler et~al.(2012)Hutzenthaler, Jentzen, Kloeden,
  et~al.]{hutzenthaler2012strong}
Martin Hutzenthaler, Arnulf Jentzen, Peter~E Kloeden, et~al.
\newblock Strong convergence of an explicit numerical method for {SDE}s with
  nonglobally {L}ipschitz continuous coefficients.
\newblock \emph{The Annals of Applied Probability}, 22\penalty0 (4):\penalty0
  1611--1641, 2012.

\bibitem[Hutzenthaler et~al.(2016)Hutzenthaler, Jentzen, and
  Salimova]{hutzenthaler2016strong}
Martin Hutzenthaler, Arnulf Jentzen, and Diyora Salimova.
\newblock Strong convergence of full-discrete nonlinearity-truncated
  accelerated exponential {E}uler-type approximations for stochastic
  {K}uramoto-{S}ivashinsky equations.
\newblock \emph{arXiv preprint arXiv:1604.02053}, 2016.

\bibitem[Jentzen and Pu{\v{s}}nik(2015)]{jentzen2015strong}
Arnulf Jentzen and Primo{\v{z}} Pu{\v{s}}nik.
\newblock Strong convergence rates for an explicit numerical approximation
  method for stochastic evolution equations with non-globally {L}ipschitz
  continuous nonlinearities.
\newblock \emph{arXiv preprint arXiv:1504.03523}, 2015.

\bibitem[Jentzen and Pu{\v{s}}nik(2016)]{jentzen2016exponential}
Arnulf Jentzen and Primo{\v{z}} Pu{\v{s}}nik.
\newblock Exponential moments for numerical approximations of stochastic
  partial differential equations.
\newblock \emph{Stochastics and Partial Differential Equations: Analysis and
  Computations}, pages 1--53, 2016.

\bibitem[Jentzen et~al.(2017)Jentzen, Salimova, and Welti]{jentzen2017strong}
Arnulf Jentzen, Diyora Salimova, and Timo Welti.
\newblock Strong convergence for explicit space-time discrete numerical
  approximation methods for stochastic {B}urgers equations.
\newblock \emph{arXiv preprint arXiv:1710.07123}, 2017.

\bibitem[Kloeden and Neuenkirch(2007)]{kloeden2007pathwise}
Peter~E Kloeden and Andreas Neuenkirch.
\newblock The pathwise convergence of approximation schemes for stochastic
  differential equations.
\newblock \emph{LMS journal of Computation and Mathematics}, 10:\penalty0
  235--253, 2007.

\bibitem[Renardy and Rogers(2006)]{renardy2006introduction}
Michael Renardy and Robert~C Rogers.
\newblock \emph{An introduction to partial differential equations}, volume~13.
\newblock Springer Science \& Business Media, 2006.

\bibitem[Runst and Sickel(1996)]{runst1996sobolev}
Thomas Runst and Winfried Sickel.
\newblock \emph{Sobolev spaces of fractional order, Nemytskij operators, and
  nonlinear partial differential equations}, volume~3.
\newblock Walter de Gruyter, 1996.

\bibitem[Sell and You(2013)]{sell2013dynamics}
George~R Sell and Yuncheng You.
\newblock \emph{Dynamics of evolutionary equations}, volume 143.
\newblock Springer Science \& Business Media, 2013.

\end{thebibliography}

\end{document}